\documentclass[review,hidelinks,onefignum,onetabnum]{siamart251216}

\usepackage{lipsum}
\usepackage{amsfonts}
\usepackage{graphicx}
\usepackage{epstopdf}
\usepackage{algorithmic}
\ifpdf
  \DeclareGraphicsExtensions{.eps,.pdf,.png,.jpg}
\else
  \DeclareGraphicsExtensions{.eps}
\fi

\newsiamremark{remark}{Remark}
\newsiamremark{hypothesis}{Hypothesis}
\crefname{hypothesis}{Hypothesis}{Hypotheses}
\newsiamthm{claim}{Claim}
\newsiamremark{fact}{Fact}
\crefname{fact}{Fact}{Facts}

\headers{X. YUAN AND G. WU}{\sc RANDOMIZED KRYLOV SUBSPACE METHODS FOR SHIFTED LINEAR SYSTEMS}

\title{Randomized Krylov Subspace Methods for Shifted Linear Systems}

\author{Xiang Yuan\thanks{School of Mathematics, China University of Mining and Technology, 221116, Jiangsu, P.R. China.}
\and Gang Wu\thanks{Corresponding author. School of Mathematics,
China University of Mining and Technology, Xuzhou, 221116, Jiangsu, P.R. China.
E-mail: {\tt gangwu@cumt.edu.cn}. This work is supported by the National Natural Science Foundation of China under Grants 12671452 and 12271518.}}

\usepackage{amsopn}

\usepackage{graphicx}
\usepackage{subfigure}
\usepackage{amsmath}
\usepackage{booktabs} 
\usepackage{multirow}
\usepackage{array}
\usepackage{slashed}
\usepackage{amsmath,bm}
\usepackage{appendix}
\usepackage{epsfig}
\usepackage{fancyhdr}
\usepackage{graphicx}
\usepackage{multirow}
\usepackage{float}
\usepackage{amsmath,amssymb}
\usepackage{graphicx}
\usepackage{algorithm}
\usepackage{tcolorbox}
\usepackage{algorithmic}
\usepackage{tabularx} 
\usepackage{ragged2e} 
\usepackage{booktabs} 
\usepackage{tablefootnote}
\ifpdf
\hypersetup{
  pdftitle={Randomized Krylov Subspace Methods for Shifted Linear Systems},
  pdfauthor={X. YUAN AND G. WU}
}
\fi

\begin{document}
\nolinenumbers

\maketitle

\begin{abstract}
We propose several novel Krylov subspace algorithms termed shifted randomized GMRES (RandGMRES-Sh), shifted randomized FOM (RandFOM-Sh), shifted randomized GMRES with deflated restarting (RandGMRES-Sh-DR), and shifted randomized FOM with deflated restarting (RandFOM-Sh-DR), for solving shifted linear systems. These methods are designed to significantly reduce the computational overhead of orthogonal basis generation while preserving orthogonality under specific conditions. Besides, this unified framework makes it possible to solve all shifted linear systems at the same time by using a single shared basis. These new approaches combine multiple advanced techniques together effectively, such as randomized sketching, the control of collinear residual directions, and deflated restarting methods. Compared with the standard versions of traditional algorithms, the new algorithms show a steady but moderate improvement in convergence speed. Furthermore, we theoretically prove that although GMRES and sketched GMRES have different algorithmic solution steps, their iterates are all confined to and together span the same basic Krylov subspace. We also work out the theoretical relationship between the solutions produced by RandGMRES-Sh and RandFOM-Sh at the same iteration step before the first restart. Moreover, we strictly prove how the solutions of RandGMRES-Sh and RandFOM-Sh are related to those of their traditional versions, GMRES-Sh and FOM-Sh, respectively.
\end{abstract}

\begin{keywords}
randomized sketching, shifted systems, Krylov methods, GMRES, FOM, Deflated restarting.
\end{keywords}

\begin{MSCcodes}
68Q25, 68R10, 68U05
\end{MSCcodes}

\section{Introduction}
Let $A \in  \mathbb{C}^{n\times n}$ be nonsingular and nonhermitian in general and let $\sigma_j \in \mathbb{C}$ be such that the shifted matrix $\hat{A}_j=A+ \alpha_i I,\ i=1,2,\dots,\ell$ is nonsingular, too. We
consider the two systems
\begin{equation}
\hat{A}_i\mathbf{x}_i=\mathbf{b}     \qquad    i=1,2,\dots,\ell, \label{1.1}
\end{equation} 
 for several parameter $ \alpha_i $. In subsequent discussions, if $i=1$, system \eqref{1.1} will be referred to as the seed system, else system \eqref{1.1} will be termed the additional system and $l$ represents the number of additional system.
 In computational science and engineering, many key problems need to efficiently solve linear systems with shifted structures like (\ref{1.1}), where the parameter $ \alpha_i $ takes different values. Such problems are widely encountered in various fields, including the computation of multi-mass propagators for lattice fermions in quantum chromodynamics (QCD) \cite{23,24}, frequency response analysis in structural dynamics, parameter identification in control system design,\cite{36,37} and the numerical solution of time-dependent partial differential equations\cite{38,39} and exponential function of large matrices \cite{31}. Moreover, in large-scale network analysis and information retrieval, the core of the Google PageRank algorithm \cite{10} is essentially solving a specific type of linear system. Here, changes in the damping factor can also be seen as a shift parameter that affects both the convergence performance and the results of web page ranking. Traditional methods require solving each shifted system separately, which leads to a total computational cost that increases linearly with the number of shift parameters. This linear cost growth becomes a key bottleneck, especially for large-scale non-symmetric matrices, and greatly limits efficiency in scientific computing and data-intensive applications.

 The Krylov subspaces$$\mathcal{K}_{m}(\hat{A}_i, \mathbf{b})=\operatorname{span}\left\{\mathbf{b}, \hat{A}_i \mathbf{b}, \hat{A}_i^{2} \mathbf{b}, \ldots, \hat{A}_i^{m-1} \mathbf{b}\right\},\ i=1,2,\dots,\ell,$$for the seed and the add systems, respectively, are identical.
A basic approach to this problem was proposed by Andreas Frommer \cite{3}, who showed how the GMRES method could be used to make residuals collinear across all shifted systems. This work established a unified Krylov subspace framework for solving them at the same time. Building on this idea, Valeria Simoncini \cite{4} later found that the residuals produced by the FOM method are naturally collinear for shifted systems, and she proposed a new algorithm to make use of this inherent property. To solve the slow convergence of restarted GMRES, Ronald B. Morgan introduced the GMRES with deflated restarting (GMRES-DR) \cite{13,14} algorithm, which uses a deflated restarting technique. Later, he applied this improved solver within the collinear residual framework to efficiently solve shifted linear systems \cite{5}. More recently, Balabanov and Grigori developed the randomized GMRES (RandGMRES) \cite{1,19} and RandFOM \cite{2} algorithms by adding randomized sketching \cite{9} to the Gram-Schmidt orthogonalization process. Based on RandGMRES, Yongseok Jang et al. then created the Randomized FGMRES-DR \cite{8} method. In a related line of research, Nakatsukasa and Tropp proposed the sketched GMRES (sGMRES) method and a sketched Rayleigh-Ritz procedure. This further inspired Burke, G{\"u}ttel, and others to develop the sGMRES with deflated restarting (GMRES-SDR) \cite{6} algorithm.

This work makes five main contributions to the development of randomized Krylov solvers for shifted linear systems:
(1) In Section 2, we review the theoretical foundations and algorithmic frameworks of shifted GMRES (GMRES-Sh) and shifted FOM (FOM-Sh), and provide a systematic summary of sketching techniques used in Krylov subspace methods. This review lays the necessary background for the developments in the following sections.
(2) In Section 3, we enhance Andreas Frommer's GMRES-Sh \cite{3} and Valeria Simoncini's FOM-Sh \cite{4} algorithms with randomized sketching techniques \cite{12}. Through strict mathematical derivation, we develop a computationally efficient variant that keeps the original shifted solver framework.
(3) In Section 4, we extend Ronald B. Morgan's GMRES-DR \cite{14} and FOM-DR \cite{17} approaches by adding sketching methodology. Theoretical analysis shows that this extension maintains the original deflation properties while meeting the requirements of shifted solutions.
(4) In Section 5, we derive the solution relationships between GMRES-Sh and FOM-Sh, and thus establish the exact connections between their randomized versions (RandGMRES-Sh, RandFOM-Sh).
(5) In Section 6, extensive numerical experiments on PageRank, QCD and matrix exponential computation problems show that our proposed algorithms are much faster in computational performance than existing baseline methods across multiple large-scale datasets.
\section{Preliminary}\label{sec2}
In this chapter, we first review the basic frameworks of the shifted GMRES \cite{3} and shifted FOM \cite{4} algorithms. Next, we introduce the properties and selection criteria of random sketching matrices, and then use these matrices to review the random Arnoldi method \cite{1}. From the derivation process, we identify two Arnoldi relations that the method satisfies and analyze its computational cost in comparison with the standard Arnoldi method. For the existing RandGMRES \cite{1} and sGMRES \cite{7} methods, we prove their equivalence. Finally, to ensure the stability of the algorithm, we introduce the random Arnoldi method with re-orthogonalization \cite{41}, which provides a foundation for the subsequent sections of this paper.

\subsection{Shifted GMRES and Shifted FOM algorithms}
\subsubsection{A Shifted GMRES algorithm}
 In \cite{3}, Frommer and Glassner propose a shifted GMRES algorithm for shifted linear system. The key idea is to run the GMRES algorithm \cite{17,18} on the seed system and then generate approximate solutions to the additional systems imposing collinearity with the computed seed residual. However, the situation deteriorates as soon as restarts occur. More precisely, let the approximation of the seed system can be formulated as $\mathbf{x}_m^1=\mathbf{x}_0^1+V_m\mathbf{y}_m^1$. Then the GMRES residual  $\mathbf{r}_m^1$ of the seed system and the residuals $\mathbf{r}_m^{i}\ (i=2,3,\dots,\ell)$ of the additional systems will generally not be collinear anymore, so that $\mathcal{K}_{m}((A+\alpha_1I), \mathbf{b}) \ne \mathcal{K}_{m}(\hat{A}_i, \mathbf{b})$ for $i=2,3,\dots,\ell$. As a result, from the first restartion, the matrix multiplications for the additional systems can not be saved anymore.
 
 \renewcommand{\thealgorithm}{1}
\begin{algorithm}
\caption{A Restarted Shifted GMRES \cite{3}}
\begin{algorithmic}[1]\label{A11}
\REQUIRE 
   $A\in \mathbb{C} ^{n\times n} $, $\mathbf{b}\in \mathbb{C} ^{n\times 1} $, initial vector $\mathbf{x}_0=0$; maximum number of restarts $maxit$; the number of inner iterations $m$; a prescribed tolerance $tol$; the shifts $\alpha=\{\alpha_1,\alpha_2,\dots,\alpha_\ell\}$.
\ENSURE 
    $\mathbf{x}_m^i$ for $i=1,2,\dots\ \ell$.\\
   \STATE Set $\mathbf{r}_0^1=b$, $\beta_m^i=\left \| \mathbf{r}_0^1 \right \|_2$, $\mathbf{x}_m^i=\mathbf{x}_0,\ i=1,2,\dots,l$, $\mathbf{v}_1=\mathbf{r}_0^1/\beta_m^i.$
   \STATE $\bf for$ $j=1:maxit$
      \STATE~~~ [$V_m$,\ $\bar{H}_m^1$,\ $\beta$]=Arnoldi($(A+\alpha_1I)$,\ $\mathbf{v}_1$,\ m), \  $\beta$ = $\left \| \mathbf{r}_0^1 \right \|_2$.
    \STATE   ~~~$\mathbf{y}_m^1=\mathop{\arg\min}\limits_{y\in \mathbb{R}^{n}}\left \| (\beta \mathbf{e}_1-\bar{H}_m^1\mathbf{y}) \right \|_2$;
  \STATE  ~~~$\mathbf{x}_m^1=\mathbf{x}_m^1+V_m\mathbf{y}_m^1$,~~~ $\mathbf{r}_m^1=b-A\mathbf{x}_m^1-\sigma_1\mathbf{x}_m^1$,~~~ $\mathbf{z}_{m+1}=\beta \mathbf{e}_1-\bar{H}_m\mathbf{y}_m^1$, \ $\mathbf{r}_0^1=\mathbf{r}_m^1$;
     \STATE ~~~$\bf for$ $i=2:\ell$
     \STATE   ~~~~~~Solve $\begin{bmatrix}
 \widehat{H}_m^i &\mathbf{z}_{m+1}
\end{bmatrix}   \begin{bmatrix}
 \mathbf{y}_m^i\\
\beta_m^i
\end{bmatrix}=\beta_0^i\beta \mathbf{e}_1$, where $\hat{H}^i_m=\bar{H}_m+(\alpha_i-\alpha_1)\bar{I}.$
  \STATE  ~~~~~~$\mathbf{x}_m^i=\mathbf{x}_m^i+V_m\mathbf{y}_m^i$,~~~ $\mathbf{r}_m^i=b-A\mathbf{x}_m^i-\alpha_i\mathbf{x}_m^i$,~~~ $\beta_0^i=\beta_m^i$, $\mathbf{v}_1=\mathbf{r}_0^1/\beta_m^i$.
   \STATE ~~~$\bf end~for$
  \STATE    ~~~$\bf if$ $\left \| \mathbf{r}_i \right \|_2<tol$ for $i=1,2,\dots\ \ell$.
  \STATE    ~~~~~~break
  \STATE    ~~~${\bf end~if}$
  \STATE $\bf end~for$
\end{algorithmic}
\end{algorithm}
 
 In order to deal with this problem, Frommer and Glassner \cite{17} suggested forcing the residual $\mathbf{r}_m^i$ to be collinear to $\mathbf{r}_m^1$.  More precisely, assume that the residuals satisfy $\mathbf{r}_0^i=\beta_0^i\mathbf{r}_0^1\ (i=2,3,\dots,\ell)$, where $\mathbf{r}_0^1$ is the residual corresponding to the initial guess $\mathbf{x}_0^1$ of the seed system. Note that this relation can be satisfied with all $\beta_0^i=1$ for all $i$ if we take $\mathbf{x}_0^i=0$. Let the approximate solutions of the additional systems be $\mathbf{x}_m^i=\mathbf{x}_0^i+V_m\mathbf{y}_m^i\ (i=2,3,\dots,\ell)$. Then we obtain$$\mathbf{r}_0^i-(A+\alpha_iI)V_m\mathbf{y}_m^i=\beta^i_m (\mathbf{r}_0^1-AV_m\mathbf{y}_m^1),\qquad i=2,3,\dots,\ell.$$
 Denote $\beta=\left \| \mathbf{r}_0^1 \right \|_2$ and $\mathbf{z}_{m+1}=\beta \mathbf{e}_1-\bar{H}_m\mathbf{y}_m^1$, then $\mathbf{r}_{m}^1=V_{m+1}\mathbf{z}_{m+1}$ is known and the unknowns $\mathbf{y}_m^i$, $\beta_m^i$ can be calculated by solving the following linear system \cite{17} 
 $$\begin{bmatrix}
  \hat{H}_m^i&\mathbf{z}_{m+1}
\end{bmatrix} \begin{bmatrix}
 \mathbf{y}_m^i\\\beta_m^i
\end{bmatrix}=\beta_0^i\beta \mathbf{e}_1,\qquad i=2,3,\dots,\ell.$$Let $\bar{I}=\begin{bmatrix}
 I &0
\end{bmatrix}^\top$, we have $\hat{H}^i_m=\bar{H}_m+(\alpha_i-\alpha_1)\bar{I}$.
However, $\mathbf{z}_{m+1}$ will approach zero as the seed system converges. The final algorithm can be written as Algorithm \ref{A11}.

\subsubsection{A Shifted FOM algorithm}
In \cite{4}, Valeria Simoncini propose a shifted FOM algorithm for shifted linear system. The key idea is to run the FOM algorithm \cite{17} on the seed system and then generate approximate solutions to the additional systems imposing collinearity with the computed seed residual. Theoretical analysis shows that when all systems share the same initial residual $\mathbf{r}_0$ and the same basis $V_m$, their residuals $\mathbf{r}_m^i$ remain mutually collinear after one restart cycle. As a result, the Krylov subspaces satisfy $\mathcal{K}_m(A, \mathbf{b}) = \mathcal{K}_m(\hat{A}_i, \mathbf{b})$ for $i = 1, 2, \dots, l$. Accordingly, the approximate solution for each shifted system can be expressed as $\mathbf{x}_m^i = \mathbf{x}_0^i + V_m \mathbf{y}_m^i$.

Restarting can be employed in the shifted case, the key fact is that the FOM residual $\mathbf{r}_m$  is a multiple of the basis vector $\mathbf{v}_{m+1}$  that is \cite{17} $$\mathbf{r}_m=-h_{m+1,m}\mathbf{v}_{m+1}(\mathbf{y}_m)_{(m)}.$$
In the last expression, $(\mathbf{y}_m^i)_{(m)}$  is the $m$-th component of the vector. Then \cite{4} shows that collinearity 
$$\mathbf{r}_m^i=-h_{m+1,m}\mathbf{v}_{m+1}(\mathbf{y}_m^i)_{(m)}, \ i=1,2,\dots,\ell$$
still holds in the shifted case when FOM is applied. This indicates that all residuals of the shifted systems are collinear with the $(m+1)$-th basis vector.

The final algorithm can be written as Algorithm \ref{A22}.
\renewcommand{\thealgorithm}{2}
\begin{algorithm}
\caption{A Restarted Shifted FOM \cite{4}}
\begin{algorithmic}[1]\label{A22}
\REQUIRE 
   $A\in \mathbb{C} ^{n\times n} $, $\mathbf{b}\in \mathbb{C} ^{n\times 1} $, initial vector $\mathbf{x}_0=\mathbf{0}$; Maximum number of restarts $maxit$; the number of inner iterations $m$; a prescribed tolerance $tol$; the shifts $\alpha=\{\alpha_1,\alpha_2,\dots,\alpha_\ell\}$.
\ENSURE 
    $\mathbf{x}_m^i$ for $i=1,2,\dots,\ell$.\\
   \STATE Set $\mathbf{r}_0=\mathbf{b}$, $\beta_m^i=\left \| \mathbf{r}_0 \right \|_2$, $\mathbf{x}_m^i=\mathbf{x}_0,\ i=1,2,\dots,\ell$, and $\mathbf{v}_1=\mathbf{r}_0/\left \| \mathbf{r}_0 \right \|_2.$
   \STATE $\bf for$ $j=1:maxit$
      \STATE ~~~ [$V_m$,\ $\bar{H}_m$]=Arnoldi(A,\ $\mathbf{v}_1$,\ m),~~~ $H_m=\bar{H}_m(1:m,:).$
   \STATE ~~~$\bf for$ $i=1:l$;
    \STATE   ~~~~~~$\mathbf{y}_m^i=(H_m+\alpha_iI)^{-1}\mathbf{e}_1\left | \beta_m^i \right |$,~~~ $\mathbf{x}_m^i=\mathbf{x}_m^i+V_m\mathbf{y}_m^i$;
  \STATE   ~~~~~~$\mathbf{r}_m^i=b-A\mathbf{x}_m^i-\alpha_i\mathbf{x}_m^i$,~~~ $\beta_m^i=-h_{m+1,m}(\mathbf{y}_m^{i})_{(m)}$;
   \STATE ~~~$\bf end~for$
   \STATE ~~~$\mathbf{v}_1=\mathbf{v}_{m+1}$.
  \STATE    ~~~$\bf if$ $\left \| \mathbf{r}_i \right \|_2<tol$ for $i=1,2,\dots\ \ell$.
  \STATE    ~~~~~~break
  \STATE    ~~~${\bf end~if}$
  \STATE $\bf end~for$
\end{algorithmic}
\end{algorithm}

\subsection{Randomized Krylov Subspace Methods}
To make use of randomized sketching, we further assume that we have an operator $\Theta \in \mathbb{R}^{k\times n}$ with $m< k\ll n$, which acts as an approximate isometry for the Euclidean norm. Next, we provide some more detailed definitions(see, e.g., \cite{12,20,21,22}).
\begin{definition}\label{D2.1}
\cite{12} For $0 \le \varepsilon < 1$, the sketching matrix is said to be an $\varepsilon$-subspace embedding for $V\subset \mathbb{R}^n$ if we have$$\forall \mathbf{x}, \mathbf{y} \in V, \quad|\langle\mathbf{x}, \mathbf{y}\rangle-\langle\boldsymbol{\Theta} \mathbf{x}, \boldsymbol{\Theta} \mathbf{y}\rangle| \leq \varepsilon\|\mathbf{x}\|\|\mathbf{y}\|. $$
\end{definition}

Definition \ref{D2.1} indicates that the inner product of vectors can be approximated by the inner product after sketching.
\begin{definition}\label{D2.2}
\cite{12} The sketching matrix $\Theta \in \mathbb{R}^{k \times n}$ is called a $(\varepsilon ,\delta ,d)$ oblivious $l_2$-subspace embedding if it is an $\varepsilon$-embedding for any fixed $d$-dimensional subspace $V\subset \mathbb{R}^n$ with probability at least $1-\delta$.
\end{definition}

Definition \ref{D2.2} states that matrix $A$ can be constructed from a carefully chosen distribution with high probability, and it satisfies the condition for any low-dimensional subspace.
When $k$ is sufficiently large, many distributions satisfy the $( \varepsilon, \delta, d )$ independent $l_2$-subspace embedding property. Examples include the Gaussian distribution, Rademacher distribution, subsampled randomized Hadamard transform, and Fourier transform. Gaussian sketching matrices have elements independently drawn from a standard Gaussian normal distribution, but their drawback is relatively high computational cost and low efficiency.

Before proceeding, we first recall the randomized Arnoldi procedure (hereafter referred to as the RandArnoldi algorithm). We then review the randomized GMRES (RandGMRES) method originally introduced by Balabanov and Grigori \cite{1}. A key condition that RandGMRES must satisfy is the projection condition 
$$\Theta (\textbf{b}- A\textbf{x})\bot \Theta AV_m$$in the sketched subspace.
Consider solving the linear system $A\textbf{x} = \textbf{b}$, where $A\in \mathbb{R}^{n\times n}$, given an initial vector $\textbf{x}_0$, let $\textbf{r}_0 = \textbf{b} - A \textbf{x}_0$. Our objective is to solve the optimization problem:
\begin{equation}\label{2.1}
x=\mathop{\arg\min}\limits_{x\in \mathcal{K}_{m}}\left \| \Theta (A\textbf{x}-\textbf{b}) \right \|_2,
\end{equation}
where $\Theta \in \mathbb{R}^{k \times n}$ is a sketching matrix, starting with $\textbf{v}_1 = \textbf{r}_0/\left \| \Theta \textbf{r}_0 \right \|$, we compute $\textbf{w} = A\textbf{v}_1$ and $\textbf{s}_1=\Theta \textbf{v}_1$, the vector $\Theta \mathbf{w}$ is then orthogonalized against $\textbf{s}_1$ to form $\textbf{s}_2$. Given that the iteration has proceeded to step $m$ with available vectors $\{\textbf{v}_1,\textbf{v}_2,\cdots,\textbf{v}_{m}\}$ and $\{\textbf{s}_1,\textbf{s}_2,\dots,\textbf{s}_{m}\}$, we compute 
\begin{equation}\label{2.2}
\textbf{v}_{m+1}= \frac{\textbf{w}-\sum_{i=1}^{m} (\textbf{s}_i^\top \Theta \textbf{w})\textbf{v}_i}{\left \| \Theta \textbf{w}-\sum_{i=1}^{m}(\textbf{s}_i^\top \Theta \textbf{w})\textbf{s}_i \right \|_2 },
\end{equation}
$$A\textbf{v}_{m}= \sum_{i=1}^{m}(\textbf{s}_i^\top \Theta \textbf{w})\textbf{v}_i +\left \| \Theta \textbf{w}-\sum_{i=1}^{m}\textbf{s}_i(\textbf{s}_i^\top \Theta \textbf{w}) \right \|_2\textbf{v}_{m+1}.$$
We denote the coefficients preceding $\{\textbf{v}_1,\textbf{v}_2,\dots,\textbf{v}_{m}\}$ as the first $m$ elements of the $m$-th column of upper Hessenberg matrix $\bar{H}$, record the coefficient preceding $\mathbf{v}_{m+1}$ as $h_{m+1,m}$. Under this construction, the Arnoldi relation $AV_m=V_{m+1}\bar{H}$ continues to hold.

From
\begin{equation}\label{2.3}
\mathbf{s}_{m+1} = \frac{\Theta \mathbf{w}-\sum_{i=1}^{m}(\mathbf{s}_i^\top \Theta \mathbf{w})\mathbf{s}_i}{\left \| \Theta \mathbf{w}-\sum_{i=1}^{m}(\mathbf{s}_i^\top \Theta \mathbf{w})\textbf{s}_i \right \|_2 },
\end{equation}
$$\Theta A\mathbf{v}_{m} = \sum_{i=1}^{m}(\mathbf{s}_i^\top \Theta \mathbf{w})\mathbf{s}_i+\left \| \Theta \mathbf{w}-\sum_{i=1}^{m}\mathbf{s}_i(\mathbf{s}_i^\top \Theta \mathbf{w}) \right \|_2\mathbf{s}_m.$$
We observe that the vectors $\mathbf{s}_i$ and $\mathbf{v}_i$ are generated using the same coefficients from the Hessenberg matrix $\bar{H}$, which does not change during the process. As a result, the sketched Arnoldi relation $\Theta A V_m = \Theta V_{m+1} \bar{H}$ holds and the matrix $\Theta V_m$ is column-orthogonal. While Balabanov and Grigori note that computing $\mathbf{s}_i=\Theta \mathbf{v}_i$ provides better stability, our experiments show that when the number of RandArnoldi steps is small, the computational cost of alternative scheme \eqref{2.3} is lower. This makes the method especially suitable for restarted scenarios with many outer loops and few inner iterations, where the reduced cost per restart cycle leads to significant total savings.

In practical use, the basis vectors generated by the RandArnoldi process tend to lose orthogonality. To reduce this numerical problem, reorthogonalization can be used, leading to the stabilized procedure described in Algorithm \ref{A1}. The recurrence relations in \eqref{2.2} and \eqref{2.3} together define the RandArnoldi algorithm. The choice to perform reorthogonalization is based on the ratio of $\left \|\Theta \mathbf{w}  \right \|_2$ to the norm of the vector after orthogonalization, $\left \| \Theta \mathbf{w}-\sum_{1}^{m-1}\mathbf{s}_i(\mathbf{s}_i^\top \Theta \mathbf{w}) \right \|_2$, as defined in \eqref{2.3}. A threshold of $\tau$ is adopted; if the ratio is lower than this value, the operations in \eqref{2.2} and \eqref{2.3} are repeated. After that, the computed value of $(SV_{m-1})^\top \mathbf{s}_m$ is added to the corresponding column of the Hessenberg matrix $\bar{H}$.
 
From \eqref{2.1}, it follows that the least squares problem solved by RandGMRES is 
$$\mathbf{y}=\mathop{\arg\min}\limits_{\mathbf{y}\in \mathbb{C}^{m}} \left \| \Theta V_{m+1}(\bar{H}\mathbf{y}-\left \| \Theta \mathbf{r}_0 \right \|\mathbf{e}_1 ) \right \|_2.$$
and the final solution of RandGMRES is $\mathbf{x}=\mathbf{x}_0+V_m\mathbf{y}.$

In theorem \ref{T4.1}, we establish connections between the solutions of two currently popular algorithms employing sketching techniques sGMRES \cite{7} and RandGMRES \cite{1}, and the basis vectors of these two methods differ only by an upper triangular matrix. sGMRES uses the truncated Arnoldi method during the orthogonalization process, but the functions optimized by both methods are the same, these two methods are equivalent.
\begin{theorem}\label{T4.1}
The solution space generated by $V$ in RandGMRES is the same as that generated by $U$ in sGMRES. Furthermore, the two bases are related through the transformation$$U=VR,$$ where $R$ is the matrix obtained by reversing the signs of specific columns in the upper triangular factor from the QR decomposition of $\Theta U$.
\end{theorem}
\begin{proof}
The vector $\mathbf{w}_m^R = A\mathbf{v}_{m-1}$ in RandGMRES \cite{1} and $\mathbf{w}_m^S = A\mathbf{u}_{m-1}$ in sGMRES \cite{7} effectively introduce the same search direction $A^{m-1}\mathbf{r}_0$ into their respective iterative processes. As a result, the solution spaces of both methods remain confined to the same Krylov subspace $\mathcal{K}_m(A,\mathbf{r}_0)$. Based on $span(V)=span(U)$, the result for $$span( \Theta V)=span( \Theta U)$$ can be correspondingly derived.

Since $\mathbf{v}_1=\frac{\mathbf{r}_0}{\left \| \Theta \mathbf{r}_0 \right \|_2}$ and $\mathbf{u}_1=\frac{\mathbf{r}_0}{\left \| \mathbf{r}_0 \right \|_2 }$, we have $\Theta \mathbf{u}_1=Q_1\widehat{R}_1$, where $Q_1=\Theta \mathbf{v}_1$, we now proceed by mathematical induction, suppose that after $m-1$ iterations of the RandGMRES and sGMRES algorithms, the resulting basis vectors are $$V_{m-1}=\begin{bmatrix}
  \mathbf{v}_1&\mathbf{v}_2  &\dots  &\mathbf{v}_{m-1}
\end{bmatrix} \qquad and \qquad U_{m-1}=\begin{bmatrix}
  \mathbf{u}_1&\mathbf{u}_2  &\dots  &\mathbf{u}_{m-1}
\end{bmatrix}$$ respectively, 
and the relation 
$\Theta U_{m-1}=\Theta V_{m-1}\widehat{R}_{m-1}$ holds, in which $\widehat{R}_{m-1}$ is an upper triangular matrix.
If the orthonormal basis $\{
  \Theta \mathbf{v}_1,\Theta \mathbf{v}_2  ,\dots  ,\Theta \mathbf{v}_{m-1}
\}$ is extended to an orthonormal basis $$\{
  \Theta \mathbf{v}_1 ,\dots, \Theta \mathbf{v}_{m-1} ,\Theta \mathbf{v}_{m}
\}$$ of the subspace spanned by $\{
  \Theta \mathbf{v}_1 ,\dots, \Theta \mathbf{v}_{m-1} ,\Theta \mathbf{w}_m
\}$, and the vector $\mathbf{w}_m=A^{m-1}r_0$ represents the basis vector generated at the $m$-th iteration, which is common to both the RandGMRES and sGMRES algorithms. Then the new basis vector $\Theta \mathbf{v}_{m}$ exists and is uniquely determined up to sign, and $\mathbf{a}$ and $b$ are the orthogonalization coefficients of $\Theta w$ relative to $\Theta V_{m-1}$.
The result of the m-th iteration has the following relationship
$$\Theta U_m=\begin{pmatrix}
  \Theta V_{m-1}&\Theta v_{m}
\end{pmatrix}\begin{pmatrix}
  \widehat{R}_{m-1}&\mathbf{a} \\
  0&b
\end{pmatrix}.$$
According to mathematical induction, the following relationship holds for $U$ and $V$ of arbitrary dimensions
\begin{equation}\label{111111}
  \Theta U=\Theta V
diag(\pm 1, \pm1, \dots, \pm1)
diag(\pm 1, \pm, \dots, \pm1)
R.
\end{equation}
Suppose the QR decomposition of $\Theta U$ is $\Theta U=\widetilde{Q} \widetilde{R}$, it follows from \eqref{111111} that
$$\widetilde{Q}=\Theta V diag(\pm 1, \pm, \dots, \pm1) \qquad and \qquad \widetilde{R}= diag(\pm 1, \pm, \dots, \pm1)R.$$
However, we know there must exist a matrix $\hat{R}$ such that $U = V \hat{R}$. This matrix $\hat{R}$ must satisfy the sketched subspace relation $\Theta U=\Theta V\hat{R}$. Since $\Theta V$ is column-orthogonal, this system admits a unique solution, yielding the explicit relation $\hat{R}=R$.
\end{proof}

To ensure the orthogonality of $\Theta V_{m+1}$, we present Algorithm \ref{A1}. This is basically consistent with the re-orthogonalization method proposed in \cite{41}
\renewcommand{\thealgorithm}{3}
\begin{algorithm}
\caption{A RandArnoldi with Reorthogonalization \cite{41}}
\begin{algorithmic}[1]\label{A1}
\REQUIRE 
   Matrix $A\in \mathbb{C} ^{n\times n}$, non-zero vector $\mathbf{r}_0\in \mathbb{C}^{n \times 1}$, sketching matrix $\Theta \in \mathbb{C}^{s \times n}$, the number of iterations $m$;
\ENSURE 
    $\Theta V_{m+1}\in \mathbb{C}^{s \times m}, V \in \mathbb{C}^{n \times m}$, $\bar{H} \in \mathbb{C}^{(m+1) \times m}$.\\
   \STATE $SV(:,1) = \Theta \mathbf{r}/ \left \| \Theta \mathbf{r} \right \|_2;$
   \STATE $V(:,1) = \mathbf{r}/ \left \| \Theta \mathbf{r} \right \|_2;$ 
   \STATE $\bf for$ $i=2:m+1$
    \STATE   ~~~$\mathbf{w}=A V(:, i-1),~~~ \mathbf{s w}=\Theta \mathbf{w},~~~ n o r m w=\|\mathbf{s w}\|_2$;
  \STATE  ~~~$\bar{H}(1: i-1, i-1)=(SV)^{\top}(\mathbf{s w})$;
  \STATE   ~~~$\mathbf{s w}=\mathbf{s w}-SV \bar{H}(1: i-1, i-1),~~~ \mathbf{w}=\mathbf{w}-V \bar{H}(1: i-1, i-1);$
  \STATE    ~~~$\bf if$  $\left \|  \mathbf{sw} \right \|/normw< \tau$
  \STATE    ~~~~~~$\bar{H}(1: i-1, i-1)=\bar{H}(1: i-1, i-1)+(SV)^{\top}(\mathbf{s w}),$
  \STATE    ~~~~~~$\mathbf{s w}=\mathbf{s w}-SV \bar{H}(1: i-1, i-1) ,~~~ \mathbf{w}=\mathbf{w}-V \bar{H}(1: i-1, i-1);$
  \STATE    ~~~${\bf end~if}$
  \STATE   ~~~$ \bar{H}(i,i-1)=\left \| \mathbf{sw} \right \|_2;$
  \STATE   ~~~$SV=[SV \ \mathbf{sw}]$, $V=[V \ \mathbf{w}];$
  \STATE    ~~~$\bf if$ $\bar{H}(i,i-1)=0$
  \STATE    ~~~~~~stop
  \STATE    ~~~${\bf end~if}$
  \STATE $\bf end~for$
\end{algorithmic}
\end{algorithm}

This section analyzes the difference in floating-point operation counts between Balabanov and Grigori's scheme $\mathbf{s}_i=\Theta \mathbf{v}_i$ and the alternative scheme \eqref{2.3}. When implemented using the Fast Walsh-Hadamard Transform, the former requires only $\mathcal{O}(2n\log_{2}{(k+1)})$ flops \cite{1}. In contrast, at the $m$-th  step, scheme \eqref{2.3} involves a linear combination of $m$ vectors of dimension $k$ resulting in $(2m+1)k$ flops. Our analysis demonstrates that for relatively small values of $m$, scheme \eqref{2.3} is computationally faster than the $\mathbf{s}_i=\Theta \mathbf{v}_i$ approach.

Next, we will discuss the difference in floating-point operation counts when both the Arnoldi process and the RandArnoldi process are run $m$ times.

1. The RArnoldi process performs sketching on m column vectors, whereas the Arnoldi process does not. We estimate its floating-point operation count to be $\mathcal{O}(2mn\log_{2}{(k+1)})$.

2. In the RandArnoldi process, the Hessenberg matrix $\bar{H}$ is obtained by computing inner products of the sketched column vectors, and we estimate its floating-point operation count to be $\frac{(2k-1)(m+1)m}{2}$. In contrast, in the Arnoldi process, the Hessenberg matrix is derived from inner products of the original (unsketched) column vectors, and we estimate its floating-point operation count to be $\frac{(2n-1)(m+1)m}{2}$.

3. The RandArnoldi process requires the computation of \eqref{2.3}. We estimate its floating-point operation count to be $m(m+2)k$.
\begin{remark}
We can estimate that the difference in floating-point operation \\counts between the two methods is $$(n-k)(m+1)m-m(m+2)k-\mathcal{O}(2mn\log_{2}{(k+1)}).$$ There is a relatively small integer $a$ such that for $m > a$, the RandArnoldi process is faster than the standard Arnoldi process. The performance gap between them becomes larger as $m$ increases.
\end{remark}

Then \eqref{2.1} can be transformed into $$\mathbf{y}=\mathop{\arg\min}\limits_{\mathbf{y}\in \mathbb{C}^{n}}\left \| \bar{H}\mathbf{y}-\left \| \Theta \mathbf{r}_0 \right \|\mathbf{e}_1  \right \|_2.  $$

Finally, the transformed system can be solved via QR factorization-computed using either Givens or Householder transformations. In this setting, the RandGMRES solution, expressed as $\mathbf{x}=\mathbf{x}_0+V_m\mathbf{y} $, yields a residual norm for the original system that closely approximates that produced by the GMRES method.
\section{Randomized Krylov Subspace Methods for Solving Shifted Linear Systems}\label{sec3}
This section presents two restarted Krylov subspace algorithms for solving shifted linear systems. The proposed framework improves both the GMRES-Sh and FOM-Sh methods by incorporating randomized sketching techniques, with its core methodology consisting of two key components:(1) Applying the RandGMRES iterations exclusively to the seed system. (2) Enforcing collinearity between the residual of the additional system and that of the seed system, which is achieved by projecting the additional system's solution update onto the same subspace. This ensures that the entire family of shifted systems is solved efficiently within a unified framework.
\subsection{A Restarted RandGMRES for shifted linear system}
Here, assume that the seed system is set as follows:$$(A+\alpha_1)\mathbf{x}_1=\mathbf{b}.$$
Suppose there are a total of $(\ell-1)$ additional systems and set up the additional systems as follows:$$(A+\alpha_iI)\mathbf{x}_i=\mathbf{b},\ i=2,3,\dots,\ell.$$
For the Randomized GMRES algorithm of  Balabanov and Grigori, we analyze its residual:
\begin{equation}\label{aa}
  \mathbf{r}_m^1=V_{m+1}(\left \| \Theta \mathbf{r}_0^1 \right \|_2\mathbf{e}_1-\bar{H}_m^1\mathbf{y}^1_m ).
\end{equation}
Here, we denote $\left \| \Theta \mathbf{r}_0 \right \|_2\mathbf{e}_1-\bar{H}_m^1\mathbf{y}_m$ as $\mathbf{z}_{m+1}$.
Since the Arnoldi relation $AV_m=V_{m+1}\bar{H}$ holds, the Arnoldi relation for the displacement system is \begin{equation}\label{ab}
(A+\alpha_i I)V_m=V_{m+1} \hat{H}_m^i
\end{equation} still holds, denote it as $\bar{\alpha_i}=\alpha_i-\alpha_1$ and $\hat{H}_m^i=H_m^1+\bar{\alpha}_i
\bar{I}$
, $i=2,\dots,\ell$.

Next we consider the properties of the basis obtained from the RandArnoldi process. Theorem \ref{T2.3} shows that if the basis is orthogonal after dimensionality reduction via sketching, it must be inherently linearly independent.
\begin{theorem}\label{T2.3}
	Suppose $V_{m+1} \in \mathbb{R}^{ n \times (m+1)}$ is a set of bases with $m+1<n$, and $\Theta$ is a $\varepsilon$-subspace embedding for $V_{m+1}$. If $\Theta V_{m+1}$ is orthogonal, then $V_{m+1}$ is column linearly independent.
\end{theorem}
\begin{proof}
	An $m+1$ dimensional matrix $V_{m+1}=(\mathbf{v}_1,\dots,\mathbf{v}_{i-1},\mathbf{v}_i,\mathbf{v}_{i+1},\dots,\mathbf{v}_{m+1})$, for $\forall \mathbf{v} \in span\{\mathbf{v}_1,\dots,\mathbf{v}_{i-1},\mathbf{v}_{i+1},\dots,\mathbf{v}_{m+1}\}$, we have that $\Theta \mathbf{v}_i$ is orthogonal to $\Theta \mathbf{v}$. According to definition \ref{D2.1}, we have
	$$\left \| \left \langle \mathbf{v}_i,\mathbf{v} \right \rangle-\left \langle \Theta \mathbf{v}_i,\Theta \mathbf{v} \right \rangle   \right \| \le \varepsilon \left \| \mathbf{v}_i \right \|\cdot \left \| \mathbf{v} \right \|.$$
Since $\Theta \mathbf{v}$ and $\Theta \mathbf{v}_i$ are orthogonal to each other
	$$\left \| \left \langle \mathbf{v}_i,\mathbf{v} \right \rangle   \right \| \le \varepsilon \left \| \mathbf{v}_i \right \|\cdot \left \| \mathbf{v} \right \|.$$  
Based on the geometric properties inherent in the inner product of vectors
	$$\left \| \mathbf{v}_i \right \|\cdot \left \| \mathbf{v} \right \|\left | \cos \theta  \right |  \le \varepsilon \left \| \mathbf{v}_i \right \|\cdot \left \| \mathbf{v} \right \|. $$   
Given that $\left \| \mathbf{v} \right \| >0 $ and $\left \| \mathbf{v_i} \right \|>0$, we are allowed to subtract $\left \| \mathbf{v} \right \|,\left \| \mathbf{v_i} \right \|$ from both sides of the inequality
	$$\left | \cos \theta  \right | \le \varepsilon<1.$$
Then $\mathbf{v}_i$ cannot be linearly represented by$\{\mathbf{v}_1,\dots,\mathbf{v}_{i-1},\mathbf{v}_{i+1},\dots,\mathbf{v}_{m+1}\}$.
\end{proof}

Since $\Theta V_{m+1}$ is column-orthogonal, we may consider setting $\Theta \mathbf{r}_{m}^{i}=\beta_m^{i} \Theta \mathbf{r}_m^{1}$; however, $\Theta \mathbf{r}_{m}^{i}=\beta_m^{i} \Theta \mathbf{r}_m^{1}$ is not equivalent to $\mathbf{r}_m^{i}=\beta_m^{i}\mathbf{r}_m^{1}$. This directly affects the construction of the solution space in subsequent iterations. 

Next we consider how to force the collinearity of the residual directions of the forcing addition system and the seed system after reatarting. The key idea is to set
$$\mathbf{r}_m^{i}=\beta_m^{i}\mathbf{r}_m^{1}, \qquad i=2,3,\dots,\ell, $$
where $\mathbf{r}_m^{i}$ denotes the residual of the $i$-th additional system at the $m$-th iteration, and $\beta_m^{i}$ denotes the multiple of the seed system residual by the residual of the $i$-th additional system at the $m$-th iteration. We assume the initial value of $\beta_0^{i}$ is 1, and assign $\beta_m^{i}$ to $\beta_0^{i}$ after each iteration, Notice that
\begin{equation}\label{aa1}
  b-(A+\alpha_i I)\mathbf{x}_{m-1}^{i}-(A+\alpha_i I)V_m\mathbf{y}_m^{i}=\beta_m^{i}(\mathbf{b}-(A+\alpha_1 I)\mathbf{x}_{m-1}^{1}-(A+\alpha_1 I)V_m\mathbf{y}_m^{1}).
\end{equation}
It follows from \eqref{ab} and \eqref{aa1} that
\begin{equation}\label{aa2}
  \beta_0^{i}\mathbf{r}_0^{1}-V_{m+1}\hat{H}_m^i\mathbf{y}_m^{i}=\beta_m^{i}(\mathbf{r}_0^{1}-V_{m+1}\bar{H}_m^1\mathbf{y}_m^{1}).
\end{equation}
Since $\mathbf{r}_0^{1}$ can be expressed in terms of $V_{m+1}$, the variable $V_{m+1}$ can be factored out from both sides of the equation. It follows from \eqref{aa} and \eqref{aa2} that
\begin{equation}\label{bb}
  V_{m+1}(\beta_0^{i} \\ \left \| \Theta \mathbf{r}_0^{1} \right\|_2\mathbf{e}_1-\hat{H}_m^i\mathbf{y}_m^{i})    =\beta_m^{i}V_{m+1}\mathbf{z}_{m+1}.
\end{equation}
According to Theorem \ref{T2.3}, it is evident that the solution to problem \eqref{bb} is also the solution to problem $$\hat{H}_m^i\mathbf{y}_m^{i} \\ +\beta_m^{i}\mathbf{z}_{m+1}=\beta_0^{i}\left \| \Theta \mathbf{r}_0^{1} \right \|_2\mathbf{e}_1.$$The RGMRES-DR method is detailed in Algorithm \ref{A2}
\renewcommand{\thealgorithm}{4}
\begin{algorithm}
\caption{A Restarted shifted RandGMRES}
\label{A2}
\begin{algorithmic}[1]
\REQUIRE 
    $A\in \mathbb{C} ^{n\times n} $, $\mathbf{b}\in \mathbb{C} ^{n\times 1} $, sketching matrix $\Theta \in \mathbb{C}^{s \times n}$; Maximum number of restarts $maxit$; the number of inner iterations $m$; a prescribed tolerance $tol$; the shifts $\alpha=\{\alpha_1,\alpha_2,\dots,\alpha_\ell\}$.
\ENSURE 
    $\mathbf{x}_m^i, \mathbf{r}_m^i$.\\
   \STATE Set  $\mathbf{r}_0^i=\mathbf{b}$; $\mathbf{x}_0^i= zeros(n,1)$; $\beta_0^i= 1$; $i=1,2,\dots,\ell$.
   \STATE Select the system $A+\alpha_pI,\ p=1$ as the seed system;
   \STATE $\bf for$ $j=1:maxit$
   \STATE   ~~~$[\Theta V_{m+1}, V_{m+1},\ \bar{H}_m^p]=RandArnoldi(A+\alpha_pI,\ \mathbf{r}_0^p,\ \Theta,\ m)$;
   \STATE   ~~~$\mathbf{y}_m^p= \bar{H}_m^p \setminus \left \| \Theta \mathbf{r}_0^p \right \| \mathbf{e}_1$,~~~$\mathbf{x}_m^p=\mathbf{x}_0^p+V_{m}\mathbf{y}_m^p$,~~~$\mathbf{z}_m=\left \| \Theta \mathbf{r}_0^p \right \| \mathbf{e}_1-\bar{H}_m^p\mathbf{y}_m^p$,~~~$\beta_0^p=1.$
  \STATE   ~~~${\bf for}$ $\alpha_i \subset \alpha,\ i \ne p$;
  \STATE    ~~~~~~$\hat{H}_m^i= \bar{H}_m^p+(\alpha_i-\alpha_p)\begin{bmatrix}
 I &0
\end{bmatrix}^\top$;
  \STATE    ~~~~~~$\begin{pmatrix}
 \hat{H}_m^i&\mathbf{z}_m
\end{pmatrix}\begin{pmatrix}
 \mathbf{y}_m^i\\ \beta_m^i
\end{pmatrix}=(\frac{\beta_0^i}{\beta_0^p}) \left \| \Theta \mathbf{r}_{0}^p \right \|\mathbf{e}_1$;
  \STATE    ~~~~~~$\mathbf{x}_m^i=\mathbf{x}_0^i+V_{m}\mathbf{y}_m^i$,~~~~$\mathbf{x}_0^i=\mathbf{x}_m^i$,~~~~$\mathbf{r}_m^i=\mathbf{b}-A\mathbf{x}_m^i-\alpha_i\mathbf{x}_m^i$,~~~~$res_i=\left \| \mathbf{r}_m^i \right \|_2$,~~~~$\beta_0^i=\beta_m^i$;
  \STATE    ~~~${\bf end~for}$
  \STATE    ~~~$\bf if$  $max(res)>tol$
  \STATE    ~~~~~~Find the maximum Euclidean norm in $res$, find the index p, $A+\sigma_pI$ is seed system;
  \STATE    ~~~~~~$\mathbf{r}_0^p= \mathbf{r}_m^p$;
  \STATE    ~~~$\bf else$ 
  \STATE    ~~~~~~break
  \STATE    ${\bf end~for}$
\end{algorithmic}
\end{algorithm}

Finally, solving the solution of the additional system is equivalent to solving the solution of the following system of equations:
\begin{equation}\label{2.4}
	\begin{pmatrix}
		\hat{H}_m^i&\mathbf{z}_{m+1}
	\end{pmatrix}\begin{pmatrix}
		\mathbf{y}_m^{i}\\
		\beta_m^{i}
	\end{pmatrix} =\beta_0^{i} \left \| \Theta \mathbf{r}_0^{1} \right \|_2 \mathbf{e}_1.
\end{equation}
$\beta_{m}^{i}$ represents the multiple of the residual of the seed system for the $i$-th additional system at the $m$-th iteration. Next, perform a QR decomposition on the coefficient matrix of \eqref{2.4}, and solve the system of equations according to the back-substitution method. Given the initial values, the solution of the shifted system at the $m$-th iteration is finally expressed as $\mathbf{x}_m^{i} = \mathbf{x}_0^{i}+V_m\mathbf{y}_m^{i}$. 
In this work, we use the seed selection strategy similar to other works \cite{26,27} to improve the performance of the proposed RandGMRES-Sh method for solving the shifted systems. 
During each cycle, we select the seed system with the maximum residual norm among the nonconverged shifts. 
We choose the seed system such that
$$\left \| \mathbf{r}_m^p \right \| =\max_{i=1,2,\dots,\ell }\left \| \mathbf{r}_m^i \right \|.$$
 We remove the system from the set of shift systems if it converges to the required accuracy. This entire process is repeated until all shifted systems converge.
\subsection{A Restarted RandFOM for shifted linear system}\label{s3.2}
For the Randomized FOM algorithm, we refer to this algorithm as RandFOM for short.
The condition that RandFOM \cite{2} must satisfy is
\begin{equation}\label{2a}
 \Theta (\mathbf{b}- A\mathbf{x})\bot \Theta V_m,
\end{equation}
 which is sketched Petrov-Galerkin condition and $V_m$ is an m-dimensional basis obtained from the RandArnoldi algorithm. To implement RandFOM, we need to reformulate the RandArnoldi relation. It follows from \eqref{2.2} and \eqref{2.3} that 
 \begin{equation}\label{2aa}
  AV_m=V_{m}H_m+\mathbf{v}_{m+1}\mathbf{h}_{m+1}
 \end{equation}
 and 
 \begin{equation}\label{2aaa}
   \Theta AV_m=\Theta V_{m}H_m+\Theta \mathbf{v}_{m+1}\mathbf{h}_{m+1}.
 \end{equation}
 We denote $\bar{H}_m=\begin{pmatrix}
 H_m\\\mathbf{h}_{m+1}
\end{pmatrix}$. According to \eqref{2a} we continue to derive the system of equations that need to be solved after projection
 $$(\Theta V_m)^\top \Theta (\mathbf{r}_0-AV_m\mathbf{y}_m)=0.$$
 It follows from \eqref{aa} and \eqref{2aaa} that
 $$(\Theta V_m)^\top\Theta \mathbf{r}_0-(\Theta V_m)^\top \Theta (V_{m}H_m+\mathbf{v}_{m+1}\mathbf{h}_{m+1})\mathbf{y}_m=0.$$
 Based on the orthogonality of $\Theta V_{m+1}$, we derive 
 \begin{equation}\label{2aaaa}
   H_m\mathbf{y}_m=\left \| \Theta \mathbf{r}_0 \right \|_2 \mathbf{e}_1.
 \end{equation}
Therefore $\mathbf{y}_m=(H_m)^{-1} \left \| \Theta \mathbf{r}_0 \right \|_2 \mathbf{e}_1$ and $\mathbf{x}_m=\mathbf{x}_0+V_m\mathbf{y}_m$ is a solution to RandFOM. 

Based on the idea of collinear residuals, we analyze the residual of the RandFOM algorithm
$$\mathbf{r}_m^{i}=\mathbf{b}-(A+\alpha_i I)\mathbf{x}_0^{i}-(A+\alpha_i I)V_m\mathbf{y}_m^{i}.$$
It follows from \eqref{2aaa} that
$$\mathbf{r}_m^{i}=\mathbf{r}_0^{i}-V_m H_m^i\mathbf{y}_m^{i}-(\mathbf{h}_{m+1}^i\mathbf{y}_m^{i})\mathbf{v}_{m+1}.$$
Here we denote $\bar{H}_m^i=\begin{pmatrix}
 H_m^i\\\mathbf{h}_{m+1}^i
\end{pmatrix}$.
Since $\mathbf{r}_0^i$ can be expressed in terms of $V_{m+1}$, we can obtain
$$\mathbf{r}_m^{i}=V_m(\left \| \Theta \mathbf{r}_0^{i} \right \|_2\mathbf{e}_1-H_m^i\mathbf{y}_m^{i})-(\mathbf{h}_{m+1}^i\mathbf{y}_m^{i})\mathbf{v}_{m+1}.$$
If the solution process for all shifted systems uniformly uses the RandFOM algorithm, we derive that $\left \| \Theta \mathbf{r}_0^{i} \right \|_2\mathbf{e}_1-H_m\mathbf{y}_m^{i}=0$. 
Thus, the residuals of all shifted systems are collinear $$\mathbf{r}_m^ {i}=-(\mathbf{h}_{m+1}^i\mathbf{y}_m^{i})\mathbf{v}_{m+1}$$ and the residual is a constant multiple of the vector $\mathbf{v}_{m+1}$. The detailed implementation of the RandFOM-Sh method is provided in Algorithm \ref{A3}.

We can clearly see that the RandFOM-Sh algorithm basically uses the RandFOM algorithm for all shifted systems. Specifically, it only creates one set of basis vectors $V_m$ with orthonormal columns (in the sense of randomized sketching) for all shifted systems being solved at the same time. Sharing this basis will not reduce the performance of the algorithm.
\renewcommand{\thealgorithm}{5}
\begin{algorithm}
\caption{A Restarted Shifted RandFOM}
\label{A3}
\begin{algorithmic}[1]
\REQUIRE 
    $A\in \mathbb{C} ^{n\times n} $, $\mathbf{b}\in \mathbb{C} ^{n\times 1} $, sketching matrix $\Theta \in \mathbb{C}^{s \times n}$; Maximum number of restarts $maxit$; the number of inner iterations $m$; a prescribed tolerance $tol$; the shifts $\alpha=\{\alpha_1,\alpha_2,\dots,\alpha_\ell\}$.
\ENSURE 
    $\mathbf{x}_m^i, \mathbf{r}_m^i$.\\
   \STATE   $\mathbf{r}_0^i=\mathbf{b}$, $\mathbf{v}=\mathbf{b}$, $\mathbf{x}_m^i= zeros(n,1)$, $\beta^i= \left \| \Theta \mathbf{r}_0 \right \|_2$; $i=1,2,\dots,\ell$.
  \STATE   ${\bf for}$ $j=1:maxit$
  \STATE   $[\Theta V_{m+1}, V_{m+1},\bar{H}_m]=RArnoldi(A,\mathbf{v},m,\Theta)$;
  \STATE   $H= \bar{H}(1:m,:)$; $h= \bar{H}(m+1,m)$; $\mathbf{v}=V_{m+1}(:,m+1)$;  
  \STATE   ~~~${\bf for}$ $i=1:\ell$
  \STATE   ~~~~~~$\mathbf{y}_m^i=(H+\alpha_iI)^{-1}  \left | \beta^i \right |   \mathbf{e}_1$;
  \STATE   ~~~~~~$\beta^i=-h(\mathbf{y}_m^i)(m)$; 
  \STATE   ~~~~~~$\mathbf{x}_m^i=\mathbf{x}_m^i+V_{m}\mathbf{y}_m^i$; $\mathbf{r}_m^i=\mathbf{b}-A\mathbf{x}_m^i-\alpha_i\mathbf{x}_m^i$;
  \STATE   ~~~${\bf end~for}$
  \STATE    ~~~$\bf if$ $\left \| \mathbf{r}_i \right \|_2<tol$ for $i=1,2,\dots\ \ell$.
  \STATE    ~~~~~~break
  \STATE    ~~~${\bf end~if}$ 
  \STATE   ${\bf end~for}$
\end{algorithmic}
\end{algorithm}
\section{Randomized Krylov Subspace Method with Deflated Restarting for Solving Shifted Systems of Equations}\label{sec4}
\subsection{Randomized Krylov subspace method with Deflated Restarting}\label{sec4.1}
This section discusses a simple implicitly restarted method.
In standard restarted GMRES, once the restart limit is reached, the entire Krylov subspace is discarded. This often leads to slow convergence because important spectral components related to small eigenvalues usually slow to converge are lost. To fix this problem, we use a deflated restart strategy. In this strategy, we explicitly keep a set of approximate eigenvectors when restarting, then project the solution onto the subspace orthogonal to these eigenvectors. This deflation process not only speeds up convergence, especially for ill-conditioned problems, but also reduces the solver's sensitivity to the choice of restart dimension, thus improving both robustness and computational efficiency.
Following the development of GMRES, Morgan introduced the GMRES with Deflated Restarting (GMRES-DR) method \cite{14}. During restart, it retains a set of approximate eigenpairs \cite{15,16} and uses them to augment or enrich the new Krylov subspace constructed after the restart. This ensures that the newly generated subspace is not only spanned by the initial residual but also incorporates the most significant spectral information from the previous iteration cycle. Building upon Morgan's GMRES-DR and FOM-DR frameworks, this section presents two novel randomized variants: (1)Randomized GMRES with Deflated Restarting (RandGMRES-DR). (2)Randomized FOM with Deflated Restarting (RandFOM-DR).
\subsubsection{A RandGMRES with Deflated Restarting}
Similar to GMRES-DR \cite{14} , the proposed RandGMRES-DR method must meet two basic conditions:

The first condition is RandGMRES-DR's residual needs to satisfy\cite{1}
\begin{equation}\label{3.1}
\Theta \mathbf{b}-\Theta A\mathbf{x} \bot \Theta AV_m.
\end{equation}
In the first cycle, we consider the sketched residuals in \eqref{3.1} which is the residual of the RandGMRES
$$\Theta \mathbf{r}=\Theta \mathbf{b}-\Theta A\mathbf{x}=\Theta \mathbf{r}_0-\Theta AV_m\mathbf{y}=\Theta V_{m+1}(\left \| \Theta \mathbf{r}_0 \right \|\mathbf{e}_1-\bar{H}_m\mathbf{y} )\in span\{\Theta V_{m+1}\}. $$

The second condition is the harmonic Ritz pair needs to satisfy (\ref{3.2})\cite{6}
\begin{equation}\label{3.2}
\widetilde{\lambda} \Theta \mathbf{\widetilde{x}}-\Theta A\mathbf{\widetilde{x}} \bot \Theta AV_m,
\end{equation}
where $(\widetilde{\lambda}, \mathbf{\widetilde{x}})$ denotes the approximate eigenpair obtained through the randomized Arnoldi process. The sketched residual of this eigenvalue problem needs to satisfy a condition of 
$$\Theta\widetilde{\mathbf{r}}=\widetilde{\lambda} \Theta \mathbf{\widetilde{x}}-\Theta A\mathbf{\widetilde{x}}=\widetilde{\lambda} \Theta V_m\mathbf{\widetilde{y}}-\Theta AV_m\mathbf{\widetilde{y}}=\Theta V_{m+1}(\widetilde{\lambda} \bar{I}\mathbf{\widetilde{y}} -\bar{H}_m\mathbf{\widetilde{y}})\in span\{\Theta V_{m+1}\}.$$

In the first cycle, execute RandGMRES once to obtain the orthonormal matrix $V_{m+1}$ with respect to sketching. There are two relationships that hold $\Theta AV_{m}=\Theta V_{m+1}\bar{H}_m$ and $AV_{m}=V_{m+1}\bar{H}_m$.

Next, the algorithm enters the first restart, where we utilize the basis obtained from the previous iteration to approximate the eigen-information. It follows from \eqref{3.2} that

$$(\Theta AV_{m})^\top  (\Theta \widetilde{\lambda} \mathbf{\widetilde{x}}-\Theta A\mathbf{\widetilde{x}})=0$$ 
By setting $\mathbf{\widetilde{x}}_0=0$, and utilizing $\mathbf{\widetilde{x}}=\mathbf{\widetilde{x}}_0+V_m\mathbf{\widetilde{y}}$ and $\Theta A V_m = \Theta V_{m+1} \bar{H}$, we can derive the following expression 
$$\bar{H}_m^\top (\Theta V_{m+1})^\top  (\Theta \widetilde{\lambda} V_m\mathbf{\widetilde{y}}-\Theta V_{m+1}\bar{H}_m\mathbf{\widetilde{y}})=0.$$
Let $\bar{I}=\binom{I}{0}$, since $\Theta V_{m}=\Theta V_{m+1}\bar{I}$, we are able to derive
$$\bar{H}_m^\top (\Theta V_{m+1})^\top  \Theta V_{m+1}(\widetilde{\lambda} \bar{I}\mathbf{\widetilde{y}} -\bar{H}_m\mathbf{\widetilde{y}})=0.$$
Since $\Theta V_{m+1}$ is orthogonal, we can infer
$$(\widehat{H}_m^\top \widehat{H}_m+h_{m+1,m}^2(\mathbf{e}_{m+1}  \mathbf{e}_{m+1}^\top))\mathbf{\widetilde{y}} =\widetilde{\lambda} \widehat{H}_m^\top \mathbf{\widetilde{y}}.$$

Let $\widehat{H}_m^\top =\bar{H}_m^\top  \bar{I}$  and solve the Ritz pair by solving the above mentioned generalized eigenvalue problem. Let $\{\widetilde{\mathbf{y}}_1,\widetilde{\mathbf{y}}_2, \dots, \widetilde{\mathbf{y}}_k\}$ be the eigenvectors corresponding to the k smallest eigenvalues.

Since $\Theta \mathbf{r}$ and $\Theta\widetilde{\mathbf{r}}$ both belong to $span\{\Theta V_{m+1}\}$ and both are orthogonal to $\Theta AV_{m}$. Then the residual directions of $\Theta \mathbf{r}$ and $\Theta\hat{\mathbf{r}}$ are the same. It follows from \eqref{3.1} and \eqref{3.2} that
\begin{equation}\label{3.3a}
 \Theta AV_{m}\widetilde{\mathbf{y}}_{i}=\widetilde{\lambda}_i\Theta V_m\widetilde{\mathbf{y}}_i+\beta _i\Theta \mathbf{r}.
\end{equation}

Since \eqref{3.3a} holds for $i = 1,2, \dots ,k$, we obtain
\begin{equation}\label{3.3b}
  \Theta AV_m(\widetilde{\mathbf{y}}_1,\widetilde{\mathbf{y}}_2,\dots ,\widetilde{\mathbf{y}}_k)=\Theta V_m(\widetilde{\lambda}_1\widetilde{\mathbf{y}}_1,\widetilde{\lambda} _2\widetilde{\mathbf{y}}_2,\dots,\widetilde{\lambda}_k\widetilde{\mathbf{y}}_k)+(\beta_1\Theta r, \beta_2\Theta \mathbf{r} ,\dots  ,\beta_3\Theta \mathbf{r} ).
\end{equation}

Apply the Gram - Schmidt (GS) process with reorthogonalization to the $k$ Ritz vectors to obtain $W$, thus $span\{\widetilde{\mathbf{y}}_1,\widetilde{\mathbf{y}}_2, \dots ,\widetilde{\mathbf{y}}_k\}=span\{W\}$. It follows from \eqref{3.3b} that
$$\Theta AV_mW\subset span\{\Theta V_mW,\ \Theta \mathbf{r}\},$$
Where $\Theta \mathbf{r}=\Theta V_{m+1}(\left \| \Theta \mathbf{r}_0 \right \|\mathbf{e}_1-\bar{H}_m\mathbf{y} )=\Theta V_{m+1}\mathbf{z} \in span\{\Theta V_{m+1}\}$. Let $\bar{W}=\begin{bmatrix}
  W^\top &0
\end{bmatrix}^\top$ to obtain $\bar{\mathbf{z}}$, make $\mathbf{z}$ orthogonal to $\bar{W}$ to obtain $\bar{\mathbf{z}}$. Therefore, we can assume that there exists a matrix $\bar{H}'_m$ such that 
$$\Theta AV_mW=\Theta V_{m+1}
   \begin{pmatrix}
  \bar{W}&\bar{\mathbf{z}}
\end{pmatrix}\bar{H}'_m$$ holds.
The orthogonalization of $\mathbf{z}$ with respect to $\bar{W}$ is equivalent to the orthogonalization of $\Theta \mathbf{r}$ with respect to $\Theta V_mW$.  Since both $\Theta V_{m+1}$ and $\begin{pmatrix}
  \bar{W}&\bar{\mathbf{z}}
\end{pmatrix}$ are column-orthogonal, we can derive that
$$\bar{H}'_m=\begin{pmatrix}
  \bar{W}&\bar{\mathbf{z}}
\end{pmatrix}^\top(\Theta V_{m+1})^\top \Theta AV_mW=\begin{pmatrix}
  \bar{W}&\bar{\mathbf{z}}
\end{pmatrix}^\top \bar{H}_mW.$$

Let $V_{k+1}=V_{m+1}\begin{pmatrix}
  \bar{W}&\bar{\mathbf{z}}
\end{pmatrix}$. We need to record the orthogonalization coefficient 
$\mathbf{c}$ of $\bar{\mathbf{z}}$ with respect to $\bar{W}$, where $\Theta \mathbf{r}=\Theta V_{k+1}\mathbf{c}$ and $\mathbf{r}=V_{k+1}\mathbf{c}$. 

Then, apply the RandArnoldi iteration from this point to update the rest of $V_{m+1}$. Finally, theorem \ref{T3.1} we propose demonstrates that the Arnoldi relations $AV_m=V_{m+1}\bar{H}_m$ and $\Theta AV_m=\Theta V_{m+1}\bar{H}_m$ remain valid where $\bar{H}$ is upper-Hessenberg, except for a full leading $k+1$ by $k+1$ portion.
\begin{theorem}\label{T3.1}
	If the Arnoldi relation $$\Theta AV_{m}W=\Theta V_{m+1}\begin{pmatrix}
\bar{W}
&\bar{\mathbf{z}} 
\end{pmatrix}\bar{H}'_m$$ with respect to sketching holds, then $$AV_{m}W=V_{m+1}\begin{pmatrix}
\bar{W}
&\bar{\mathbf{z}} 
\end{pmatrix}\bar{H}'_m$$ also holds.
\end{theorem}
\begin{proof}
Since $span\{AV_{m}\}\subseteq span\{V_{m+1}\}$, there must exist a matrix $\bar{H}_m$ such that the equation $AV_{m}W=V_{m+1}\bar{H}_m$ holds. Furthermore, $\bar{H}_m$ must satisfy $\Theta AV_{m}W=\Theta V_{m+1}\bar{H}_m$; because $\Theta V_{m+1}$ is orthogonal, $\bar{H}_m$ is the unique solution and $\bar{H}_m=\bar{H}'_m$. Consequently, the equation $AV_{m}W=V_{m+1}\begin{pmatrix}
\bar{W}
&\bar{\mathbf{z}} 
\end{pmatrix}\bar{H}'_m$ holds.
\end{proof}

With the Arnoldi relation established, the approximate solution of RandGMRES-DR can be obtained from \eqref{3.1} as given in $$\mathbf{y}=\mathop{\arg\min}\limits_{\mathbf{y}\in \mathbb{C}^m}\left \| \Theta \mathbf{r}_0-\Theta AV_m\mathbf{y} \right \|_2=\mathop{\arg\min}\limits_{\mathbf{y}\in \mathbb{C}^m}\left \| \bar{\mathbf{c}} -\bar{H}_m\mathbf{y} \right \|_2    $$

Let $\bar{\mathbf{c}}=\begin{bmatrix}
 \mathbf{c}^\top  &0
\end{bmatrix}^\top$. Starting from the second cycle, the residual obtained each time is $$\Theta \mathbf{r}=\Theta V_{m+1}(\bar{\mathbf{c}} -\bar{H}_m\mathbf{y})\in span\{ \Theta V_{m+1}\}.$$

Similar to the first cycle, the subsequent cycle utilize the same approach to construct the Arnoldi relation.

From the comparative data in Table \ref{tb5.1} and the algorithm's derivation process, we can see that our proposed RandGMRES-DR has an advantage over GMRES-SDR\cite{6}: it can maintain both the standard Arnoldi relation and its sketched variant, while taking nearly the same computational time. This helps to stably construct the RandGMRES-DR-Sh algorithm.
\renewcommand{\thealgorithm}{6}
\begin{algorithm}
\caption{A RandGMRES with deflated restarting}
\label{A4}
\begin{algorithmic}[1]
\REQUIRE 
    $A\in \mathbb{C} ^{n\times n} $, $\mathbf{b}\in \mathbb{C} ^{n\times 1} $, sketching matrix $\Theta \in \mathbb{C}^{s \times n}$; the number of inner iterations $m$; Maximum number of restarts $maxit$; a prescribed tolerance $tol$; the number of Ritz vectors $k$. The initial guesses $\mathbf{x}_0$
\ENSURE 
    $\mathbf{x}, \mathbf{r}, V_m, \Theta V_m, H, \mathbf{c}$.\\
   \STATE First cycle.$[\Theta V_{m+1}, V_{m+1},\bar{H}_m]=RArnoldi(A,\mathbf{r}_0,\Theta,m)$
   \STATE $\mathbf{y}= \bar{H}_m \setminus \left \| \Theta \mathbf{r}_0 \right \| \mathbf{e}_1$;
   \STATE $\mathbf{x}=\mathbf{x}_0+V_{m}\mathbf{y}$;
   \STATE $\mathbf{z}=\left \| \Theta \mathbf{r}_0 \right \| \mathbf{e}_1-\bar{H}_m\mathbf{y}$;
   \STATE   ${\bf for}$ $i=1:maxit$
   \STATE ~~~Solve for k generalized eigenvectors of $(\widehat{H}_m^\top \widehat{H}_m+h_{m+1,m}^2(\mathbf{e}_{m+1}\mathbf{e}_{m+1}^\top))\mathbf{y} =\lambda \widehat{H}_m^\top \mathbf{y}$ is denoted as $\{\mathbf{y}_1,\mathbf{y}_2,\dots, \mathbf{y}_k\}$, where $\widehat{H}_m^\top =\bar{H}_m^\top  \bar{I}$.
   \STATE ~~~Perform Gram-Schmidt orthogonalization on $\{\mathbf{y}_1,\mathbf{y}_2\dots \mathbf{y}_k\}$ to obtain $W$.
   \STATE ~~~Let $\mathbf{z}$ be orthogonal to $\bar{W}=\begin{bmatrix}
  W^\top &0
\end{bmatrix}^\top$ to obtain $\bar{\mathbf{z}}$, let $W=\begin{pmatrix}\bar{W}&\bar{\mathbf{z}}\end{pmatrix}$, and record the orthogonalization coefficient $\mathbf{c}$.
   \STATE ~~Let $\bar{H}_m'=\begin{pmatrix}
\bar{W}&\bar{\mathbf{z}} 
\end{pmatrix}^\top \bar{H}_mW$, $V_{k+1}=V_{m+1}\begin{pmatrix}
\bar{W}&\bar{\mathbf{z}} 
\end{pmatrix}$ and $\Theta V_{k+1}=\Theta V_{m+1}\begin{pmatrix}
\bar{W}&\bar{\mathbf{z}} 
\end{pmatrix}$.
\STATE ~~~Compute $A\mathbf{v}_{k+1}$, and continue the RandArnoldi process on $V_{k+1}$ until $V_m$ and $\bar{H}_m$ are obtained.
\STATE ~~~Compute $\mathbf{y}=\underset{\mathbf{y}\in \mathbb{C}^m }{argmin}\left \| \bar{\mathbf{c}} -\bar{H}_m\mathbf{y} \right \|_2$.
\STATE ~~~$\mathbf{x}=\mathbf{x}+V_m\mathbf{y}$, and if $\left \| \mathbf{r} \right \|_2 < tol$, then break.
\STATE   ${\bf end~for}$
\end{algorithmic}
\end{algorithm}

\begin{table}[H]
\scriptsize
  \centering
  \caption{Comparison of several defalted restarted Algorithms, tol=1e-8, Inner Iterations=50}%
  \label{b5.1}
    \begin{tabular}{p{1.5cm}|p{3.5cm}|p{1.0cm}|p{1.5cm}}
     \hline
      Datasets   &Method&CPU(s)
      &Outer Iterations
    \\ \hline
    
    \multirow{8}{*}{\tt SiO2}
    & GMRES-SDR\newline GMRES-DR\newline FOM-DR
    & 14.76\newline 36.55\newline 33.05
    & 19\newline 16\newline 15
    \\ \cline{2-2}
    & \textbf{RandGMRES-DR}\newline \textbf{RandFOM-DR}
    & \textbf{16.70}\newline \textbf{20.42}
    & \textbf{18}\newline \textbf{20}
    \\ \hline

    \multirow{8}{*}{\tt torso3}
    & GMRES-SDR\newline GMRES-DR\newline FOM-DR
    & 6.96\newline 17.32\newline 18.65
    & 9\newline 8\newline 9
    \\ \cline{2-2}
    & \textbf{RGMRES-DR}\newline \textbf{RFOM-DR}
    & \textbf{7.93}\newline \textbf{8.01}
    & \textbf{10}\newline \textbf{9}
    \\ \hline
    
    \multirow{8}{*}{\tt CO}
    & GMRES-SDR\newline GMRES-DR\newline FOM-DR
    & 27.49\newline 56.69\newline 63.11
    & 34\newline 23\newline 23
    \\ \cline{2-2}
    & \textbf{RGMRES-DR}\newline \textbf{RFOM-DR}
    & \textbf{31.15}\newline \textbf{24.80}
    & \textbf{32}\newline \textbf{26}
    \\ \hline
    \end{tabular}%
\end{table}%
We begin by evaluating the performance of the proposed RandGMRES-DR algorithm compared with the GMRES-SDR method. The GMRES-SDR method combines a truncated Arnoldi process with sketching techniques, and we test both methods on several large-scale datasets.
 We choose the number of harmonic Ritz vectorss to be 10, the number of inner iterations to be 50, and the accuracy to be a residual 2-norm of $10^{-8}$. As shown by the data in Table \ref{b5.1}, although the proposed RandGMRES-DR algorithm has similar computational time to the GMRES-SDR algorithm \cite{6} when solving linear systems, the analysis in Section \ref{sec4} shows that our RandGMRES-DR method, together with RandFOM-DR, provides a simpler way to ensure residual collinearity. Under this setting, all linear systems meet the corresponding Arnoldi relation.

\subsubsection{A RandFOM with Deflated Restarting}
Similar to FOM-DR, the proposed RandFOM-DR implementation must satisfy two fundamental conditions:

The first condition is RandFOM-DR's residual needs to satisfy
\begin{equation}\label{3.3}
\Theta \mathbf{b}-\Theta A\mathbf{x} \perp \Theta V_m.
\end{equation}

The second condition is the harmonic Ritz pair needs to satisfy
\begin{equation}\label{3.4}
\Theta \widetilde{\lambda} \widetilde{\mathbf{x}}-\Theta A\widetilde{\mathbf{x}} \perp \Theta V_m.
\end{equation}

In the first cycle, we consider the residuals in \eqref{3.3} which is the residual of the RandFOM. It is follows from \eqref{2aaa} that
\begin{footnotesize}
\begin{equation}
  \Theta \mathbf{r}=\Theta \mathbf{r}_0-\Theta AV_m\mathbf{y}=\Theta V_m(\left \| \Theta \mathbf{r}_0 \right \|_2\mathbf{e}_1-H_m\mathbf{y} )-\Theta \mathbf{v}_{m+1}(\mathbf{h}_{m+1} \mathbf{y})=-(\mathbf{h}_{m+1} \mathbf{y}) \Theta \mathbf{v}_{m+1}.
\end{equation}
\end{footnotesize}

we consider the residuals in \eqref{3.4}
$$\Theta \widetilde{\mathbf{r}} =\widetilde{\lambda} \Theta \widetilde{\mathbf{x}}-\Theta A\widetilde{\mathbf{x}}=\widetilde{\lambda} \Theta V_m\widetilde{\mathbf{y}}-\Theta AV_m\widetilde{\mathbf{y}}.$$

Since $\Theta \widetilde{\mathbf{r}} \perp \Theta V_m$, we derive
$$\Theta \widetilde{\mathbf{r}}= \Theta V_m(\widetilde{\lambda} \widetilde{\mathbf{y}}-H_m\widetilde{\mathbf{y}})-\Theta \mathbf{v}_{m+1}(\mathbf{h}_{m+1}\widetilde{\mathbf{y}})=-(\mathbf{h}_{m+1}\widetilde{\mathbf{y}})\Theta \mathbf{v}_{m+1}.$$

Since the residuals of \eqref{3.3} and \eqref{3.4} are both aligned with $\mathbf{v}_m$, we obtain the following relation:
$$\Theta AV_m\widetilde{\mathbf{y}}_i=\widetilde{\lambda}_i\Theta V_m\widetilde{\mathbf{y}}_i+\beta_i\Theta \mathbf{v}_{m+1}.$$for every $i = 1,2, \dots ,k$.

Apply the Gram-Schmidt (GS) process to $k$ Ritz vectors to obtain $W$. Thus\\ 
$\Theta AV_{m}W \subset span\{\Theta V_mW,\ \Theta \mathbf{v}_{m+1}\}$.

There must exist an $\bar{H}_{k+1}$ such that $$\Theta AV_{m}W=\Theta V_{m+1}\hat{W}\bar{H}_{k+1}.$$
Since the columns of $\Theta V_{m+1}$ and $\hat{W}=diag(W,1)$ are orthogonal, we can conclude that
$$\bar{H}_{k+1}=\hat{W}^\top (\Theta V_m)^\top \Theta AV_{m+1}W.$$
It is assumed that $$\bar{H}_m=(\Theta V_m)^\top \Theta AV_{m+1},$$ update $$\bar{H}_{k+1}=\hat{W}^\top \bar{H}_mW \qquad and \qquad  V_{k+1}=V_{m+1}\hat{W}.$$
Apply the Arnoldi iteration from this point to form the rest of $V_{m+1}$. Similar to RandGMRES-DR, RandFOM-DR also has two RandArnoldi relations holding: $\Theta AV_{m}=\Theta V_{m+1}\bar{H}_m$ and  $AV_{m}=V_{m+1}\bar{H}_m$. 

According to \eqref{3.3} the approximate solution of RandFOM-DR is denoted as $\mathbf{x}=\mathbf{x}_0+V_m\mathbf{y}$, where $\mathbf{y}$ is obtained by solving the upper Hessenberg system:
\begin{equation}\label{3.5}
H_m\mathbf{y}=\left \| \Theta \mathbf{r}_0 \right \|_2\mathbf{e}_{k+1}.
\end{equation}
$\mathbf{e}_{k+1}$ denotes the column vector whose $k+1$-th entry equals 1 and all other entries are 0.

Starting from the second cycle, the residual obtained each time is
$$\Theta \mathbf{r}=\Theta V_m(\left \| \Theta \mathbf{r}_0 \right \|_2\mathbf{e}_{k+1}-H_m\mathbf{y} )+\Theta (\mathbf{h}_{m+1}^\top \mathbf{y})\mathbf{v}_{m+1}.$$
According to equation \eqref{3.3}, the previous term is zero, so the residual vector is parallel to $\mathbf{v}_{m+1}$. Like the first cycle, we use the same method in the following cycles to build the Arnoldi relation. The detailed steps of the RGMRES-DR method are shown in Algorithm \ref{A5}.

\renewcommand{\thealgorithm}{7}
\begin{algorithm}
\caption{A RandFOM with deflated restarting}
\label{A5}
\begin{algorithmic}[1]
\REQUIRE 
    $A\in \mathbb{C} ^{n\times n} $, $\mathbf{b}\in \mathbb{C} ^{n\times 1} $, sketching matrix $\Theta \in \mathbb{C}^{s \times n}$; the number of inner iterations $m$; Maximum number of restarts $maxit$; a prescribed tolerance $tol$; the number of Ritz vectors $k$. The initial guesses $\mathbf{x}_0$.
\ENSURE 
    $\mathbf{x}, \mathbf{r}, V_m, \Theta V_m, H, \mathbf{c}$.\\
   \STATE $[\Theta V_{m+1}, V_{m+1},\bar{H}]=RandArnoldi(A,\mathbf{r}_0,\Theta,m);$
   \STATE $H_m=\bar{H}(1:m,:);$
   \STATE $\mathbf{y}_j= H_m\setminus \left \| \Theta \mathbf{r}_0 \right \|_2 \mathbf{e}_1;$
   \STATE $\mathbf{x}=\mathbf{x}_0+V_{m}\mathbf{y};$
   \STATE   ${\bf for}$ $i=1:maxit$
   \STATE ~~~Solve for k eigenvectors of $H_m\mathbf{y} =\lambda \mathbf{y}$ is denoted as $\{\mathbf{y}_1,\mathbf{y}_2, \dots, \mathbf{y}_k\};$
   \STATE ~~~Perform Gram-Schmidt orthogonalization on $\{\mathbf{y}_1,\mathbf{y}_2, \dots, \mathbf{y}_k\}$ to obtain $W$.
   \STATE ~~~The extension of $\bar{W}=\begin{bmatrix}
  W^\top &0
\end{bmatrix}^\top$ to $\hat{W}=diag(W,1)$.
   \STATE ~~~Let $\bar{H}'=\hat{W}^\top \bar{H}W$, $V_{k+1}=V_{m+1}\hat{W}$ and $\Theta V_{k+1}=\Theta V_{m+1}\hat{W}$.
\STATE ~~~Compute $A\mathbf{v}_{k+1}$, and continue the RandArnoldi process on $V_{k+1}$ until $V_m$ and $H_m$ are obtained
\STATE ~~~Compute $H_m\mathbf{y}=\left \| \Theta \mathbf{r}_0 \right \|_2\mathbf{e}_{k+1}$.
\STATE ~~~$\mathbf{x}=\mathbf{x}+V_m\mathbf{y}$, and if $\left \| \mathbf{r} \right \|_2 < tol$, then break.
\STATE   ${\bf end~for}$
\end{algorithmic}
\end{algorithm}
\subsection{Randomized Krylov Subspace Method with Deflated Restarting for Solving Shifted Systems of Equations}
Based on the analysis in Section \ref{sec4.1}, we have established that both RandGMRES-DR and RandFOM-DR continue to satisfy the Arnoldi relation after being augmented with harmonic Ritz pairs.
Within this framework, a key algorithmic step is to enforce the collinearity of residuals across all shifted systems. The following outlines the procedure.
\subsubsection{A Shifted RandGMRES with Deflated Restarting}
Section \ref{sec3}    explained how to implement restarted RandGMRES to solve multiple shifted systems at the same time. For the deflated restarted RandGMRES method (RandGMRES-DR), we can also solve multiple shifted systems together. From the derivation in Section \ref{sec4.1}, we know that RandGMRES-DR follows a similar Arnoldi-type recurrence. However, the resulting matrix $\bar{H}_m$ does not have the upper Hessenberg structure. Even so, this difference does not affect the collinearity of the residuals among the additional systems.

This is the same as RandGMRES-Sh, except that it uses RandGMRES-DR instead of RandGMRES, and the form of $\bar{H}_m^1$ is different.
In a single cycle, we choose the system $$(A+\alpha_1I)\mathbf{x}_1=\mathbf{b}$$ as the seed system. We denote the Arnoldi relation for the seed system as equation $(A+\alpha_1 I)V_m=V_{m+1}\bar{H}_m^1$. $V_m$ also serves as the basis for the additional systems, and $\hat{H}_m^i=\bar{H}_m^1+\bar{\alpha}_i
\bar{I}$
, $i=2,3,\dots,\ell$ is the associated upper Hessenberg matrix except for a full leading $k+1$ by $k+1$ portion. Here, we denote $\bar{\alpha_i}=\alpha_i-\alpha_1$.

We assume that the residuals satisfy the relation $\mathbf{r}_0^i=\beta_0^i\mathbf{r}_0^1$ for $i=2,3,\dots,\ell$, where $\mathbf{r}_0^1$ denotes the residual corresponding to the initial guess $\mathbf{x}_0^1$ of the seed system. It is worth noting that this relation can be satisfied by setting $\beta_0^i=1$ for all $i$ when we take $\mathbf{x}_0^i=0$. After applying RandGMRES to the seed system, a critical requirement for obtaining approximate solutions to the additional systems is to meet that the residual of the $i$-th additional system is collinear with that of the seed system
\begin{equation}\label{3.2a}
  \mathbf{r}_m^{i}=\beta_m^{i}\mathbf{r}_m^{1},\ i=2,3,\dots,\ell
\end{equation} with a collinearity coefficient of $\beta_m^{i}$, the $m$ denotes the number of inner iterations.
\renewcommand{\thealgorithm}{8}
\begin{algorithm}
\caption{A Shifted RandGMRES with Deflated Restarting}
\begin{algorithmic}[1]\label{A8}
\REQUIRE 
    $A\in \mathbb{C} ^{n\times n} $, $\mathbf{b}\in \mathbb{C} ^{n\times 1} $, sketching matrix $\Theta \in \mathbb{C}^{s \times n}$; the number of inner iterations $m$; Maximum number of restarts $maxit$; a prescribed tolerance $tol$; the number of Ritz vectors $k$; the shifts $\alpha=\{\alpha_1,\alpha_2,\dots,\alpha_\ell\}$.
\ENSURE 
    $\mathbf{x}_m^i, \mathbf{r}_m^i.$\\
   \STATE During the first cycle, the current seed system is $(A+\alpha_1I)\mathbf{x}_m^1=\mathbf{b}$, let all initial guesses satisfy $\mathbf{x}_0^i = 0\ and\  \mathbf{r}_0^i=\mathbf{b},\ i=1,2,\dots,\ell$. Apply RandGMRES to the seed system $[V_m,\bar{H}_m^1,\mathbf{x}_m^1,\Theta V_{m+1},\Theta \mathbf{r}_0^1]=RandGMRES(A+\alpha _1I,\mathbf{b},tol,\Theta,m,\mathbf{x}_0^1)$, $\mathbf{x}_0^1=\mathbf{x}_m^1$.
   \STATE For the additional system $(A+\alpha_iI)\mathbf{x}_m^i=\mathbf{b}$, $\hat{H}_m^i=\bar{H}_m^1+(\alpha_i-\alpha_1)\bar{I}$
$$\begin{pmatrix}
  \hat{H}_m^i&\mathbf{z}_m
\end{pmatrix}\binom{\mathbf{y}_m^i}{\beta_m^i}=\left \| \Theta \mathbf{r}_0^i \right \|_2 \mathbf{e}_1, i\ne 1.$$
Let $\beta = (\beta_m^1, \beta_m^2, \dots, \beta_m^L)$, $\mathbf{x}_m^i=\mathbf{x}_0^i+V_m\mathbf{y}_m^i$, $\mathbf{r}_m^i=\mathbf{b}-(A+\alpha_iI)\mathbf{x}_m^i$, $\mathbf{x}_0^i=\mathbf{x}_m^i$.
   \STATE Find the maximum value in $\beta$ denoted as $\beta_m^p$, the current seed system is $(A+\alpha_pI)\mathbf{x}_m^p=\mathbf{b}$.
   \STATE Using the RandGMRES with deflated restarting in Algorithm \ref{A4}, Utilize the harmonic Ritz pairs to update the basis and Hessenberg matrix for deflated restarting to obtain $\mathbf{x}_m^p$, $V_m$, $\mathbf{r}_m^p$ and $\bar{H}_m^p$, $\mathbf{x}_0^p=\mathbf{x}_m^p$.
   \STATE For the additional system $(A+\alpha_iI)\mathbf{x}_m^i=\mathbf{b}$, $\hat{H}_m^i=\bar{H}_m^p+(\alpha_i-\alpha_p)\binom{I}{0}$
$$\begin{pmatrix}
  \bar{H}_m^i&\mathbf{z}_m
\end{pmatrix}\binom{\mathbf{y}_m^i}{\bar{\beta}_m^i}=(\frac{\beta _m^i}{\beta_m^p}) \bar{\mathbf{c}}.$$
Then $\beta_m^i=\bar{\beta}_m^i,\ \mathbf{x}_m^i=\mathbf{x}_0^i+V_m\mathbf{y}_m^i$, $\mathbf{r}_m^i=\mathbf{b}-(A+\alpha_iI)\mathbf{x}_m^i$, $\mathbf{x}_0^i=\mathbf{x}_m^i$, where $i\ne p$.
   \STATE If the maximum value in the residual vector $\left \| \mathbf{r}_m^i \right \|$ is less than $tol$, the iteration terminates; otherwise, the process returns to Step 3 of the algorithm.
\end{algorithmic}
\end{algorithm}
Let the approximate solutions of the additional systems be $\mathbf{x}_m^i=\mathbf{x}_0^i+V_m\mathbf{y}_m^i\ (i=2,3,\dots,\ell)$. 
Substituting the residual from the previous step into the expressions for both the seed system and the additional system, it follows from \eqref{3.2a} that
\begin{equation}\label{3.2aa}
 \mathbf{b}-(A+\alpha_i I)\mathbf{x}_0^{i}-(A+\alpha_i I)V_m\mathbf{y}_m^{i}=\beta_m^{i}(\mathbf{b}-(A+\alpha_1 I)\mathbf{x}_0^{1}-(A+\alpha_1 I)V_m\mathbf{y}_m^{1}).
\end{equation}
Applying Theorem \ref{T2.3}, we can enforce the Arnoldi relation to \eqref{3.2aa} to obtain:
\begin{equation}\label{3.2aaa}
  \mathbf{r}_0^{i}-V_{m+1}\hat{H}^i_m\mathbf{y}_m^{i}=\beta_m^{i}(\mathbf{r}_0^{1}-V_{m+1}\bar{H}_m^1\mathbf{y}_m^{1}).
\end{equation}
Given that $\mathbf{r}_0^{1}$ forms the first column of $V_m$, it follows that $V_m$ can be extracted from equation \eqref{3.2aaa}. It follows from \eqref{3.2aaa} that
\begin{equation}\label{3.2aaaa}
  V_{m+1}(\beta_0^{i}
 \bar{\mathbf{c}}-\hat{H}_m^i\mathbf{y}_m^{i})=\beta_m^{i}V_{m+1}\mathbf{z}_m.
\end{equation}
According to Theorem \ref{T2.3}, it is evident that the solution to problem \eqref{3.2aaaa} is also the solution to problem
$$
  \hat{H}^i_m\mathbf{y}_m^{i}+\beta_m^{i}\mathbf{z}_m=\beta_0^{i}\bar{\mathbf{c}}.
$$
The detailed implementation of the RGMRES-DR-Sh method is provided in Algorithm \ref{A8}.

When enforcing collinear residuals across the shifted systems, the linear independence of the column vectors in $V_m$ (as established in Theorem \ref{T2.3}) serves as a crucial prerequisite. During the orthogonalization of $\mathbf{z}_m$ against $W$, the resulting coefficients must be preserved as a vector $\mathbf{c}$ for subsequent derivations. We name this complete approach the RandGMRES-DR-Sh method.

\subsubsection{A Shifted RandFOM with Deflated Restarting}
Observe the residual of the RandFOM-DR method for the system of equations with shifts
\begin{align*}
	\mathbf{r} &=\mathbf{r}_0-(A+\alpha_i I)V_m\mathbf{y}\\ &=V_m(\left \| \Theta \mathbf{r}_0 \right \|\mathbf{e}_{k+1}-H_m^i\mathbf{y} )-(\mathbf{h}_{m+1} \mathbf{y})\mathbf{v}_{m+1}.
\end{align*}
\renewcommand{\thealgorithm}{9}
\begin{algorithm}
\caption{A Shifted RandFOM with Deflated Restarting}
\begin{algorithmic}[1]\label{A9}
\REQUIRE 
    $A\in \mathbb{C} ^{n\times n} $, $\mathbf{b}\in \mathbb{C} ^{n\times 1} $, sketching matrix $\Theta \in \mathbb{C}^{s \times n}$; the number of inner iterations $m$; Maximum number of restarts $maxit$; a prescribed tolerance $tol$; the number of Ritz vectors $k$; the shifts $\alpha=\{\alpha_1,\alpha_2,\dots,\alpha_\ell\}$.
\ENSURE 
    $\mathbf{x}_m^i, \mathbf{r}_m^i.$\\
   \STATE During the first cycle, let all initial guesses satisfy $\mathbf{x}_0^i = 0\ and\  \mathbf{r}_0^i=\mathbf{b},\ i=1,2,\dots,\ell$. Apply  [$V_m$,$\bar{H}_m$]=RandArnoldi(A,$\mathbf{b}$,m).
   \STATE Let $H_m=\bar{H}_m(1:m,:)$, $H_m^i=H_m+\alpha_iI$, find the approximate solution $\mathbf{y}_m^i={(H_m^i)}^{-1}\left \| \Theta \mathbf{r}_0^i \right \|_2\mathbf{e}_{k+1}$, $\mathbf{x}_m^i=\mathbf{x}_0^i+V_m\mathbf{y}_m^i$, $\mathbf{x}_0^i=\mathbf{x}_m^i$, $\mathbf{r}_m^i=\mathbf{b}-(A+\alpha_iI)\mathbf{x}_m^i$, $i=1,2,\dots,\ell$.
   \STATE The maximum norm in $\mathbf{r}_m^i$ is denoted as $\mathbf{r}_m^p$, the current seed system is $(A+\alpha_pI)\mathbf{x}_m^p=\mathbf{b}$.
   \STATE Using the RandFOM with deflated restarting in Algorithm \ref{A5}, Using the harmonic Ritz pairs to update the basis and Hessenberg matrix for deflated restarting to obtain $\mathbf{x}_m^p$, $V_m$, $\mathbf{r}_m^p$ and $\bar{H}_m^p$, $\mathbf{x}_0^p=\mathbf{x}_m^p$.
   \STATE
   Let $H_m^i=H_m^p+(\alpha_i-\alpha_p)I$, $\mathbf{y}_m^i={(H_m^i)}^{-1}\left \| \Theta \mathbf{r}_0^i \right \|_2\mathbf{e}_{k+1}$, $\mathbf{x}_m^i=\mathbf{x}_0^i+V_m\mathbf{y}_m^i$, $\mathbf{x}_0^i=\mathbf{x}_m^i$, $\mathbf{r}_m^i=\mathbf{b}-(A+\alpha_iI)\mathbf{x}_m^i$, where $i\ne p$.
   \STATE If the maximum value in the residual vector $\left \| \mathbf{r}_m^i \right \|$ is less than $tol$, the iteration terminates; otherwise, the process returns to Step 3 of the algorithm.
\end{algorithmic}
\end{algorithm}
According to the relation in \eqref{3.5}, the residual of the shifted system is given by $\mathbf{r} = -(\mathbf{h}_{m+1} \mathbf{y})\mathbf{v}_{m+1}$, indicating that all such residuals are mutually collinear.

The detailed implementation of the RandFOM-DR-Sh method is provided in Algorithm \ref{A9}.

It can be clearly observed that the RandFOM-DR-Sh algorithm essentially applies the RandFOM-DR method collectively to all shifted systems. A key implication is that RandFOM-DR-Sh generates only a single sequence of basis vectors $ V_m $ which are column-orthogonal in the sketched sense that is shared across all shifted systems being solved simultaneously. Importantly, this shared basis strategy does not lead to any performance degradation, as information is efficiently reused without compromising numerical effectiveness.

\section{Theoretical analysis}
In this section, we present the theoretical foundations and connections among several randomized algorithms with an iterative cycling framework. Firstly, we demonstrate that the randomized shifted algorithm proposed in this paper is also a projection method. We discusses the relationship between  RandGMRES-Sh and RandFOM-Sh, and their connections with GMRES-Sh and FOM-Sh. Additionally, we have proven the relationship between the solutions of algorithms RandGMRES-Sh and GMRES-Sh, as well as between Algorithms RandFOM-Sh and FOM-Sh. 

\subsection{RandGMRES-Sh and RandFOM-Sh are projection methods}
 We build on the one-cycle description of RandGMRES-Sh and RandFOM-Sh as introduced in section \ref{sec3}.
 It is established in \cite{6} that sGMRES can be characterized as a projection method, with explicit descriptions of both residual and error projections including their range and null space independent of constraints imposed by the basis $V_m$. Building on this perspective, we further demonstrate that the additional systems in shifted linear system solved by Krylov subspace methods also admit a projection framework. Specifically, we derive the corresponding projection operators and characterize their range and null space properties.
 From Section \ref{sec3}, it can be observed that RandGMRES-Sh solves the following system when computing the approximate solution $\mathbf{y}_m^i$ for the additional system:
$$
	\begin{pmatrix}
		\hat{H}_m^i&\mathbf{z}_{m+1}
	\end{pmatrix}\begin{pmatrix}
		\mathbf{y}_m^{i}\\
		\beta_m^{i}
	\end{pmatrix} =\beta_0^{i} \left \| \Theta \mathbf{r}_0^{1} \right \|_2 \mathbf{e}_1.
$$

Moreover, since $V_{m+1}$ generated by $m$ iterations of the RandArnoldi process is of full column rank, left-multiplying both sides of the above equation by $V_{m+1}$ preserves the equality. Therefore, the solution to the additional system reduces to:
$$
	\begin{pmatrix}
		V_{m+1}\hat{H}_m^i&V_{m+1}\mathbf{z}_{m+1}
	\end{pmatrix}\begin{pmatrix}
		\mathbf{y}_m^{i}\\
		\beta_m^{i}
	\end{pmatrix} =\beta_0^{i} \mathbf{r}_0^1,
$$
It follows from (\ref{ab}) that
\begin{equation}\label{51}
 	\begin{pmatrix}
		(A+\alpha_i I)V_{m+1}&V_{m+1}\mathbf{z}_{m+1}
	\end{pmatrix}\begin{pmatrix}
		\mathbf{y}_m^{i}\\
		\beta_m^{i}
	\end{pmatrix} =\beta_0^{i} \mathbf{r}_0^1,
\end{equation}

Solve the system of equations (\ref{51}) using the Moore-Penrose inverse
$$\mathbf{y}_m^i=\begin{pmatrix}
  (A+\alpha_iI)V_m&V_{m+1}\mathbf{z}_{m+1}
\end{pmatrix}^\dagger(1:m,:) \mathbf{r}_0^1.$$

We derive the relationship between the residual after one cycle and the initial residual
$$\mathbf{r}_m^i=(I-(A+\alpha_iI)V_m\begin{pmatrix}
  (A+\alpha_iI)V_m&V_{m+1}\mathbf{z}_{m+1}
\end{pmatrix}^\dagger(1:m,:))\mathbf{r}_0^1,$$
and the approximate solution of the $i$-th additional system is
$$\mathbf{x}_m^i=\mathbf{x}_0^i+V_m\begin{pmatrix}
  (A+\alpha_iI)V_m&V_{m+1}\mathbf{z}_{m+1}
\end{pmatrix}^\dagger(1:m,:))(A+\alpha_iI)(\mathbf{x}-\mathbf{x}_0^i)=\mathbf{x}_0^i+\Pi_{\mathcal{K}_m} \mathbf{\eta}.$$
where $\mathbf{\eta}=\mathbf{x}-\mathbf{x}_0$, and $\mathbf{x}$ denotes the exact solution of the additional system.

We can straightforwardly derive expression $$\begin{pmatrix}
  (A+\alpha_iI)V_m&V_{m+1}\mathbf{z}_{m+1}
\end{pmatrix}^\dagger(1:m,:)(A+\alpha_iI)V_m=I.$$
We can prove that $$\Phi_{\mathcal{K}_m}  =(A+\alpha_iI)V_m\begin{pmatrix}
  (A+\alpha_iI)V_m&V_{m+1}\mathbf{z}_{m+1}
\end{pmatrix}^\dagger(1:m,:)$$ is a residual operator with range space $Range(\Phi_{\mathcal{K}_m}) \subseteq Range((A+\alpha_iI)V_m)$, null space $\mathcal{N}(\Phi_{\mathcal{K}_m}) \subseteq Range((\begin{pmatrix}
  (A+\alpha_iI)V_m&V_{m+1}\mathbf{z}_{m+1}
\end{pmatrix}^\dagger(1:m,:))^{\top})^{\perp}$ and $$\Pi_{\mathcal{K}_m}=V_m\begin{pmatrix}
  (A+\alpha_iI)V_m&V_{m+1}\mathbf{z}_{m+1}
\end{pmatrix}^\dagger(1:m,:))(A+\alpha_iI)$$ is a error operator with range space $Range(\Pi_{\mathcal{K}_m}) \subseteq Range(V_m)$, null space $\mathcal{N}(\Pi_{\mathcal{K}_m})\\ \subseteq Range((A+\alpha_iI)^{\top}(\begin{pmatrix}
  (A+\alpha_iI)V_m&V_{m+1}\mathbf{z}_{m+1}
\end{pmatrix}^\dagger(1:m,:))^{\top})^{\perp}$ , both of which are projection operators. 

According to the analysis in section \ref{s3.2}, RandFOM-Sh has the same effect as RandFOM when solving additive systems. Furthermore, as the analysis in \cite{3}, RandFOM is a projection method.
\subsection{The relationship between the solutions of RandGMRES-Sh and RandFOM-Sh}
For a given problem, when the GMRES and FOM algorithms span the same solution subspace, their respective iterates satisfy the following relation \cite{17} 
\begin{equation}\label{yy}
  \mathbf{x}_{m}^{G} =s_{m}^{2} \mathbf{x}_{m-1}^{G} +c_{m}^{2} \mathbf{x}_{m}^{F}
\end{equation}
and
\begin{equation}\label{yy1}
  \mathbf{r}_{m}^{G} =s_{m}^{2} \mathbf{r}_{m-1}^{G} +c_{m}^{2} \mathbf{r}_{m}^{F},
\end{equation}
where the superscripts $F$ and $G$ denote the FOM and GMRES iterates, respectively. As is well-established, the FOM-Sh method essentially applies the standard FOM algorithm individually to each shifted system. Therefore, to establish the connection between GMRES-Sh and FOM-Sh, it suffices to analyze the relationship between the solutions produced by the standard GMRES algorithm and those generated by GMRES-Sh when applied to the same shifted system.

The primary objective is to establish the relationship between the standard GMRES algorithm and the GMRES-Sh variant. Here, we assume that the seed system has zero shift, consider both methods solving the same shifted system $(A+ \sigma I)\mathbf{x}=\mathbf{b}$, where GMRES-Sh takes $A\mathbf{x} = \mathbf{b}$ as its seed system. After projection onto the appropriate subspace, the two methods respectively solve the projected systems \eqref{4.1a} and \eqref{4.1b} to generate their iterative solutions.
\begin{subequations}
  \begin{align}
  \mathbf{y}_m^G=\mathop{\arg\min}\limits_{\mathbf{y}_m^G\in \mathbb{C}^{n}} \left \| (\bar{H}_m+\alpha \bar{I})\mathbf{y}_1 -\beta \mathbf{e}_1 \right \|_2 \label{4.1a} \\
    (\bar{H}_m+\alpha \bar{I}\ \ \ \mathbf{z}_{m+1})\begin{pmatrix}
 \mathbf{y}_m^{SG}\\ \beta _0
\end{pmatrix}  =\beta \mathbf{e}_1 \label{4.1b}
  \end{align}
\end{subequations}
Let $\mathbf{y}_m^G$ and $\mathbf{y}_m^{SG}$ be the projections of GMRES and GMRES-Sh onto the basis $V_m$, respectively. Perform a full $QR$ decomposition and an economy-size $Q_mR_m$ decomposition of the shifted Hessenberg matrix: $$\bar{H}_m+\alpha \bar{I}=QR=Q_mR_m,$$ where $R_m$ is an invertible upper triangular matrix. Further, let $$Q=\left [  \mathbf{q}_1,\mathbf{q}_2,\dots ,  \mathbf{q}_{m+1} \right ] \ and \ Q^\top =\left [   \widetilde{\mathbf{q}} _1,\widetilde{\mathbf{q}} _2,\dots,   \widetilde{\mathbf{q}} _{m+1} \right ].$$
We then solve the projected systems \eqref{4.1a} and \eqref{4.1b} via an economy-size QR decomposition of the corresponding Hessenberg matrix:
$$\left\{\begin{matrix}
 \begin{pmatrix}
 R_m\\0
\end{pmatrix}\mathbf{y}_m^G=\beta \widetilde{\mathbf{q}}_1, \\
\begin{pmatrix}
\begin{matrix}
R_m\\0
\end{matrix}&Q^\top \mathbf{z}_{m+1} 
\end{pmatrix}\begin{pmatrix}
 \mathbf{y}_m^{SG}\\\beta _0
\end{pmatrix}=\beta \widetilde{\mathbf{q}}_1.
\end{matrix}\right .$$
It follows that the solution takes the form
\begin{subequations}
  \begin{align}
   \mathbf{y}_m^G&=R^{-1}_m\beta \widetilde{\mathbf{q}}_1(1:m) \label{4.2a}\\
    \begin{pmatrix}
 \mathbf{y}_m^{SG}\\\beta _0
\end{pmatrix}&=\begin{pmatrix}
  R_m^{-1}&-\frac{1}{\mathbf{q}_{m+1}^\top \mathbf{z}_{m+1}} R^{-1}_m(Q^\top \mathbf{z}_{m+1})(1:m) \\
  0&\frac{1}{\mathbf{q}_{m+1}^\top {\mathbf{z}}_{m+1}}
\end{pmatrix}\beta \widetilde{\mathbf{q}}_1. \label{4.2b}
  \end{align}
\end{subequations}
Substituting \eqref{4.2a} into \eqref{4.2b} gives
$$\mathbf{y}_m^{SG}=\mathbf{y}_m^G-\frac{\beta \widetilde{\mathbf{q}}_1(m+1)}{\mathbf{q}_{m+1}^\top {\mathbf{z}_{m+1}}} R^{-1}_mQ_m^\top \mathbf{z}_{m+1}.$$
Thus, for a given initial guess $x_0$ the following relation characterizes the solutions of GMRES and GMRES-Sh:
\begin{equation} \label{4.3}
   \mathbf{x}_m^{SG}=\mathbf{x}_0+V_m\mathbf{y}_m^{SG}=\mathbf{x}_m^G-\frac{\beta \widetilde{\mathbf{q}}_1(m+1)}{\mathbf{q}_{m+1}^\top {\mathbf{z}_{m+1}}} V_m(\bar{H}_m+\alpha \bar{I})^\dagger  \mathbf{z}_{m+1}.
\end{equation}
Finally, we can establish the relationship between the solutions of the GMRES-Sh and FOM-Sh algorithms. Substituting $x_m^G$ from \ref{4.3} into \ref{yy} yields 
\begin{align}\label{5.6}
  \mathbf{x}_m^{SG}  =&s_m^2\mathbf{x}_{m-1}^{SG}+\frac{s_m^2\beta {\mathbf{q}}_m(1)}{\mathbf{q}_{m}^\top(1:m) \mathbf{z}_{m}} V_{m-1}(\bar{H}_{m-1}+\alpha \bar{I}_{m-1})^\dagger \mathbf{z}_{m} \\& \nonumber -\frac{\beta \mathbf{q}_{m+1}(1)}{\mathbf{q}_{m+1}^\top \mathbf{z}_{m+1}} V_m(\bar{H}_m+\alpha \bar{I}_m)^\dagger \mathbf{z}_{m+1}+c_m^2\mathbf{x}_m^{SF}.
\end{align}
This leads to the following relation for the residual vectors obtained by the two methods,
\begin{align}\label{5.7}
  \mathbf{r}_m^{SG}  =&s_m^2\mathbf{r}_{m-1}^{SG}-\frac{s_m^2\beta {\mathbf{q}}_m(1)}{\mathbf{q}_{m}^\top(1:m) \mathbf{z}_{m}} V_{m}(\bar{H}_{m-1}+\alpha \bar{I}_{m-1})(\bar{H}_{m-1}+\alpha \bar{I}_{m-1})^\dagger \mathbf{z}_{m} \\& \nonumber +\frac{\beta \mathbf{q}_{m+1}(1)}{\mathbf{q}_{m+1}^\top \mathbf{z}_{m+1}} V_{m+1}(\bar{H}_m+\alpha \bar{I}_m)(\bar{H}_m+\alpha \bar{I}_m)^\dagger \mathbf{z}_{m+1}+c_m^2\mathbf{r}_m^{SF}.
\end{align}

Finally, since RandGMRES-Sh and RandFOM-Sh share the same solution framework as the original methods, and since equations (\ref{yy}) and (\ref{yy1}) continue to hold, the relationships between their solutions and residuals remain consistent with those given in (\ref{5.6}) and (\ref{5.7}).

\subsection{Relationship of the Randomized Multi-Shift Algorithms to the Conventional Multi-Shift Algorithm }
We now analyze the relationship between the solutions of the GMRES-Sh and RandGMRES-Sh algorithms established in this work before restart.
Let $V_m^{(RG)}$ and $V_m^{(G)}$ denote the basis vectors generated by the RandGMRES and GMRES algorithms after $m$ iterations, respectively. By Theorem \ref{T4.1}, the basis vectors produced by RandGMRES are equivalent to those generated by GMRES, we can express $V^{(RG)}_m$ as a linear transformation of $V^{(G)}_m$ in the form
\begin{equation}\label{4.5}
V^{(RG)}_m=V^{(G)}_mR,
\end{equation} where $R$ is an upper triangular matrix resulting from the decomposition. Let $\bar{R}=\begin{bmatrix}
 R^\top  &0
\end{bmatrix}^\top$, we may denote 
$ V_{m+1}^{(RG)}=V_{m+1}^{(G)}\bar{R}$
. The Arnoldi relation for RandGMRES and GMRES are given by 
\begin{subequations}
  \begin{align}
    AV_m^{(RG)}&=V_{m+1}^{(RG)}\bar{H}_m^{(RG)} \label{4.6a}\\
    AV_m^{(G)}&=V_{m+1}^{(G)}\bar{H}_m^{(G)} \label{4.6b}
  .\end{align}
\end{subequations}Therefore, based on the relationship between \eqref{4.6a} and \eqref{4.5}, we can obtain \begin{equation}\label{4.7}
AV_m^{(G)}=V_{m+1}^{(G)}\bar{R}\bar{H}_m^{(RG)}R^{-1}.\end{equation}From Equation \eqref{4.7} the relationship between matrix $\bar{H}_m^{(RG)}$ and $\bar{H}_m^{(G)}$ is \begin{equation} \label{4.8}
 \bar{H}_m^{(G)}=\bar{R}\bar{H}_m^{(RG)}R^{-1}. \end{equation}Next, we examine the expressions for the solutions of RandGMRES and GMRES
\begin{subequations}
  \begin{align}
       \mathbf{x}^{(RG)}_{m}&=\mathbf{x}_0+V_m^{(RG)}({{\bar{H}_m}^{(SG)}})^\dagger \left \| \theta \mathbf{r}_0 \right \| \mathbf{e}_1 \label{4.9a}\\
       \mathbf{x}^{(G)}_{m}&=\mathbf{x}_0+V_m^{(G)}({\bar{H}_m^{(G)}})^\dagger \left \|  \mathbf{r}_0 \right \| \mathbf{e}_1 .\label{4.9b}
  \end{align}
\end{subequations}

As previously discussed, based on the relationship between GMRES and GMRES-Sh derived in equation \eqref{4.3}, when GMRES, GMRES-Sh, RandGMRES, and RandG\-MRES-Sh are applied to solve the same shifted system, combining with (\ref{4.3}), the relationship between the solutions of RandGMRES-Sh and GMRES-Sh is given by the following relation
\begin{footnotesize}
\begin{align}
\mathbf{x}^{(SG)}_{m}&=\mathbf{x}^{(SRG)}_{m}+\frac{\left \|  \mathbf{r}_0 \right \| \mathbf{q}^{(G)}_{m+1}(1)}{{\mathbf{q}_{m+1}^{(G)}}^\top \mathbf{z}_{m+1}^{(G)}} V_m^{(G)}(\bar{H}_m^{(G)}+\alpha \bar{I})^\dagger  \mathbf{z}_{m+1}^{(G)}
+V_m^{(G)}({\bar{H}_m^{(G)}})^\dagger \left \|  \mathbf{r}_0 \right \| \mathbf{e}_1 \\& \nonumber -V_m^{(RG)}({\bar{H}_m^{(RG)}})^\dagger \left \| \Theta \mathbf{r}_0 \right \| \mathbf{e}_1-\frac{\left \| \Theta \mathbf{r}_0 \right \| \mathbf{q}^{(RG)}_{m+1}(1)}{{\mathbf{q}_{m+1}^{(RG)}}^\top \mathbf{z}_{m+1}^{(RG)}} V_m^{(RG)}(\bar{H}_m^{(RG)}+\alpha \bar{I})^\dagger  \mathbf{z}_{m+1}^{(RG)}
.\end{align}
\end{footnotesize}
This leads to the following relation for the residual vectors obtained by the two methods,
\begin{footnotesize}
\begin{align}
\mathbf{r}^{(SRG)}_{m}&=\mathbf{r}_{m}^{(SG)}+\frac{\left \|  \mathbf{r}_0 \right \| \mathbf{q}^{(G)}_{m+1}(1)}{{\mathbf{q}_{m+1}^{(G)}}^\top \mathbf{z}_{m+1}^{(G)}} V_{m+1}^{(G)}(\bar{H}_m^{(G)}+\alpha \bar{I})(\bar{H}_m^{(G)}+\alpha \bar{I})^\dagger  \mathbf{z}_{m+1}^{(G)} \\& \nonumber+(A+\alpha I)V_m^{(G)}({\bar{H}_m^{(G)}})^\dagger \left \|  \mathbf{r}_0 \right \| \mathbf{e}_1  -(A+\alpha I)V_m^{(RG)}({\bar{H}_m^{(RG)}})^\dagger \left \| \Theta \mathbf{r}_0 \right \| \mathbf{e}_1\\& \nonumber\frac{\left \| \Theta \mathbf{r}_0 \right \| \mathbf{q}^{(RG)}_{m+1}(1)}{{\mathbf{q}_{m+1}^{(RG)}}^\top \mathbf{z}_{m+1}^{(RG)}} V_{m+1}^{(RG)}(\bar{H}_m^{(RG)}+\alpha \bar{I})(\bar{H}_m^{(RG)}+\alpha \bar{I})^\dagger  \mathbf{z}_{m+1}^{(RG)}
.\end{align}
\end{footnotesize}
The superscript SG stands for the GMRES-Sh method, SRG for the RandGMRES-Sh method, G for the GMRES method, and RG for the RandGMRES method.

Next, we derive the relationship between the solutions of the RandFOM-Sh and FOM-Sh algorithms. It has been demonstrated by Simoncini V that in the FOM-Sh algorithm, the solutions of the seed system and the shifted systems are naturally collinear. This implies that the shifted systems essentially perform the FOM procedure. Likewise, the RandFOM-Sh algorithm effectively executes the RandFOM method when solving shifted systems. Both the RandFOM and FOM algorithms preserve the Arnoldi relations \eqref{4.6a} and \eqref{4.6b}. We denote
\begin{subequations}
  \begin{align}
       \bar{H}_m^{(G)}&=\begin{pmatrix}
{H}_m^{(F)} \\
\mathbf{h}_{m+1}^{F}
\end{pmatrix}\\
       \bar{H}_m^{(RG)}&=\begin{pmatrix}
{H}_m^{(RF)} \\
\mathbf{h}_{m+1}^{RF}
\end{pmatrix}
  \end{align}
\end{subequations}
From \eqref{4.8} we can derive the relationship 
\begin{equation}
{H}_m^{(F)}=(R{H}_m^{(RF)}+\mathbf{r}(1:m)\mathbf{h}^{(RF)}_{m+1})R^{-1}
\end{equation}\label{5.17}
 from Equations (5.16a) and (5.16b).
Next, from equation 
\begin{subequations}
  \begin{align}
       {H}_m^{(F)}\mathbf{y}_1 & =\left \| \mathbf{r}_0 \right \| \mathbf{e}_1 \\
       {H}_m^{(RF)}\mathbf{y}_2 & =\left \| \Theta \mathbf{r}_0 \right \|\mathbf{e}_1 
  \end{align}
\end{subequations} and \ref{5.17} solved by the FOM and RandFOM algorithms, we infer equation
\begin{subequations}
  \begin{align}
      \mathbf{y}_1&=R(R{H}_m^{(RF)}+\mathbf{r}(1:m)h^{(RF)}_{m+1})^{\dagger }\left \| \Theta \mathbf{r}_0 \right \| \mathbf{e}_1 \\
      \mathbf{y}_2&=({{H}_m^{(RF)}})^{-1}\left \| \mathbf{r}_0 \right \| \mathbf{e}_1  .
  \end{align}
\end{subequations}
We ultimately derive that the relationship between the solutions of RandFOM-Sh and FOM-Sh is given by relation
\begin{equation}
\mathbf{x}^{S-RF}_{m}-\mathbf{x}^{S-F}_{m}=V_m^{G}R({{H}_m^{(RF)}})^{-1}\left \| \mathbf{r}_0 \right \| \mathbf{e}_1  - V_m^{G}R(R{H}_m^{(RF)}+\mathbf{r}(1:m)h^{(RF)}_{m+1})^{\dagger }\left \| \Theta \mathbf{r}_0 \right \| \mathbf{e}_1.
\end{equation}
This leads to the following relation for the residual vectors obtained by the two methods,
\begin{align}
\mathbf{r}^{S-F}_{m}-\mathbf{r}^{S-RF}_{m}=&(A+\alpha I)V_m^{G}R({{H}_m^{(RF)}})^{-1}\left \| \mathbf{r}_0 \right \| \mathbf{e}_1  \\&\nonumber- (A+\alpha I)V_m^{G}R(R{H}_m^{(RF)}+\mathbf{r}(1:m)h^{(RF)}_{m+1})^{\dagger }\left \| \Theta \mathbf{r}_0 \right \| \mathbf{e}_1
\end{align}
\section{Numerical experiments}\label{Sec5}
This section presents a comparative performance analysis between the proposed RandGMRES-Sh, RandFOM-Sh, Rand\-GMRES-DR-Sh and RandFOM-DR-Sh algorithms and their standard counterparts
\begin{itemize}
\item[$\bullet$] Restarted Shifted GMRES (GMRES-Sh)  \cite{3}
\item[$\bullet$] Restarted Shifted FOM (FOM-Sh)  \cite{4}
\item[$\bullet$] Deflated GMRES with multiple shifts (GMRES-DR-Sh) \cite{5}
\item[$\bullet$] Deflated FOM with multiple shifts (FOM-DR-Sh) \cite{31}
\end{itemize}
across a range of numerical problems. All reported results are based on the SRFT (Subsampled Randomized Fourier Transform) sketching distribution and the size of the sketching matrix is uniformly set to $s=500$.
The numerical experiments were conducted on a workstation equipped with an AMD R7 4800H CPU running at 2.90 GHz and 16 GB of RAM, under the Windows 10 operating system. All implementations and tests were carried out in MATLAB R2021b.
To evaluate the practical efficiency of the proposed methods, the following sections provide detailed applications and comparative studies focusing on three representative problems: General Test problems, PageRank problem and Quantum Chromodynamics (QCD), The primary dataset used in this study is shown in Table \ref{tb5.1}.
\begin{table}[htbp]
\centering
\caption{Set and characteristics of the test matrix problems used in our experiments}
\label{tb5.1}
\scalebox{0.8}{\begin{tabular}{l l l l l }
\hline Example & Name & Size & Nonzeros & Problem kind \\
\bottomrule[1pt] 
 1 & SiO2 & 	155,331 & 11,283,503 & Theoretical/Quantum Chemistry Problem \\
 2 & torso3 & 	259,156 & 4,429,042 & 2D/3D Problem \\
 3 & CO & 	221,119 & 7,666,057 & Theoretical/Quantum Chemistry Problem \\
 4 & FEM\underline{~}3D\underline{~}thermal2 & 147,900 & 3,489,300 & Thermal problem \\
 5 & vfem & 93,476 & 1,434,636 & Electromagnetics Problem \\
 6 & Dubcova3 & 146,689 & 3,636,643 & 2D/3D Problem \\
 7 & eu-2005 & 862,664 & 19,235,140 & Directed Graph \\
 8 & Web-Google & 916,428 & 5,105,039 & Directed Graph \\
 9 & netherlands\underline{~}osm & 2,216,688 & 4,882,476 & Directed Graph \\
 10 & web-edu & 9,845,725 & 57,156,537 & Directed Graph \\
 11 & conf5\underline{~}4\underline{~}00l8x8\underline{~}1500 & 	49,152 & 1,916,928 & Theoretical/Quantum Chemistry Problem \\
 12 & conf5\underline{~}4-8x8-05 & 	49,152 & 1,916,928 & Theoretical/Quantum Chemistry Problem \\
 13 & conf5\underline{~}4-8x8-15 & 	49,152 & 1,916,928 & Theoretical/Quantum Chemistry Problem \\
\hline
\end{tabular}}
\end{table}

\subsection{General Test Problems}
The first set of matrices used in Examples 1, 2, 3 are extracted from the Sparse Matrix Collection \cite{28} available at the University of Florida, described in Table6.1. In our experiments, the set of shifts $\sigma =[0.0001,0.001,0.0\-1,0.1,1]$ and we deflated $k=10$ harmonic Ritz vectors. We set the maximum dimension of the search space equal to $m=50$, and $tol =10^{-6}$ and we used $maxit =10000$ for the stopping criterion parameters. 

We compare the convergence behavior, true residual norms, and CPU time of the four shifted algorithms proposed in this work with those of the four algorithms introduced in references \cite{3,4,5}.
\begin{figure}
  \centering
  \includegraphics[width=0.7\textwidth]{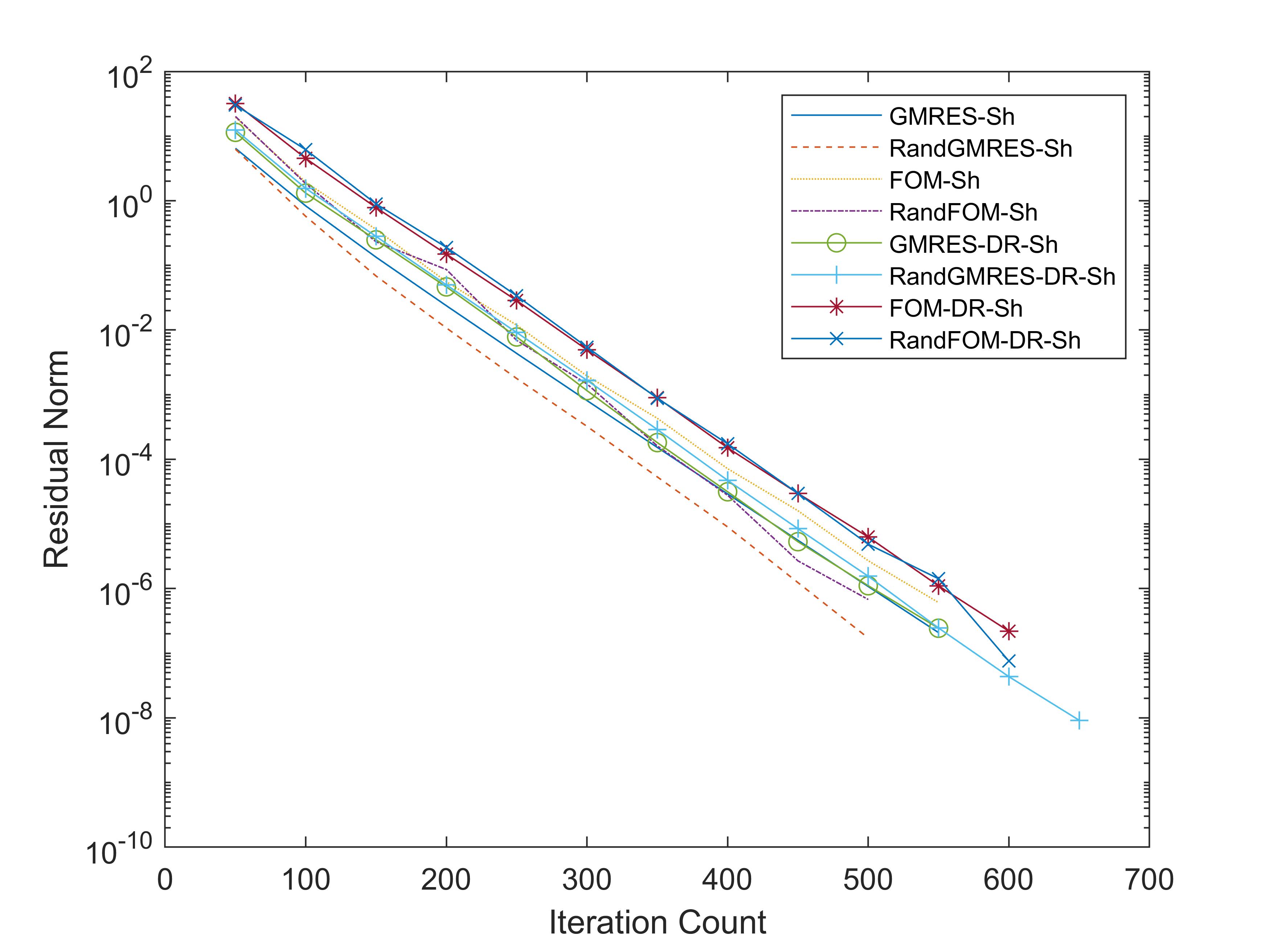}
  \caption{Convergence curve of dataset FEM\underline{~}3D\underline{~}thermal2}\label{p1}
\end{figure}
The large-scale sparse linear systems presented in the table-drawn from thermal problems, computational fluid dynamics, and electromagnetics are typically derived from the discretization of partial differential equations \cite{29,30}. In such engineering simulations, multi-scenario parameter analysis is often required, such as varying heat sources in thermal problems, diverse physical conditions in fluid dynamics, or multi-frequency excitations in electromagnetics. The shifted algorithm serves as an efficient iterative solver whose key strength lies in its ability to simultaneously handle families of linear systems with multiple right-hand sides or slight parametric variations (i.e., "shifts"). By reusing the core iterative process, the method avoids solving each parametric scenario independently, thereby significantly enhancing computational efficiency in applications such as parameter sweeps, optimization design, and uncertainty quantification. Figure \ref{p1} illustrates the convergence curves of several algorithms for matrix FEM\underline{~}3D\underline{~}thermal2 under a fixed shift.
\begin{table}[H]
\small
  \centering
  \caption{Numerical experiments on general test problems, tol=1e-6, Inner Iterations=50}%
  \label{tb5.2}
  \scalebox{0.8}{
    \begin{tabular}{p{2.7cm}|p{4.7cm}|p{1.7cm}|p{1.7cm}}
     \hline
      Datasets   &Method&CPU(s)&Outer Iterations
    \\ \hline
    
    \multirow{8}{*}{\tt FEM\underline{~}3D\underline{~}thermal2}
     & GMRES-sh\newline FOM-sh
     & 16.60\newline 15.47
     & 11\newline 11
    \\ \cline{2-2}
     & \textbf{RandGMRES-sh}\newline \textbf{RandFOM-sh}
     & \textbf{5.65}\newline \textbf{5.90}
     & \textbf{10}\newline \textbf{11}
    \\ \cline{2-4}
     & GMRES-DR-sh\newline FOM-DR-sh
     & 14.31\newline 13.81
     & 11\newline 11
    \\ \cline{2-2}
     & \textbf{RandGMRES-DR-sh}\newline \textbf{RandFOM-DR-sh}
     & \textbf{5.20}\newline \textbf{5.86}
     & \textbf{11}\newline \textbf{12}
     \\ \hline

   \multirow{8}{*}{\tt vfem}
     & GMRES-sh\newline FOM-sh
     & 12.13\newline 12.69
     & 4\newline 4
    \\ \cline{2-2}
     & \textbf{RandGMRES-sh}\newline \textbf{RandFOM-sh}
     & \textbf{7.04}\newline \textbf{7.13}
     & \textbf{4}\newline \textbf{4}
    \\ \cline{2-4}
     & GMRES-DR-sh\newline FOM-DR-sh
     & 13.35\newline 13.74
     & 5\newline 5
    \\ \cline{2-2}
     & \textbf{RandGMRES-DR-sh}\newline \textbf{RandFOM-DR-sh}
     & \textbf{7.41}\newline \textbf{7.31}
     & \textbf{5}\newline \textbf{5}
     \\ \hline

      \multirow{8}{*}{\tt Dubcova3}
     & GMRES-sh\newline FOM-sh
     & 27.93\newline 20.38
     & 4\newline 4
    \\ \cline{2-2}
     & \textbf{RandGMRES-sh}\newline \textbf{RandFOM-sh}
     & \textbf{10.45}\newline \textbf{9.74}
     & \textbf{4}\newline \textbf{4}
    \\ \cline{2-4}
     & GMRES-DR-sh\newline FOM-DR-sh
     & 9.84\newline 11.05
     & 5\newline 5
    \\ \cline{2-2}
     & \textbf{RandGMRES-DR-sh}\newline \textbf{RandFOM-DR-sh}
     & \textbf{5.57}\newline \textbf{5.62}
     & \textbf{5}\newline \textbf{5}
   \\ \hline
    \end{tabular}}
\end{table}%

\subsection{PageRank Problem}
The PageRank algorithm was first created to measure how important webpages are. It changes hyperlink networks into directed graphs and calculates eigenvectors through repeated processes to get the global ranking of webpages. As a basic tool for network analysis, this algorithm also gives a common example for studying linear algebra problems, such as sparse matrix iteration and eigenvalue computation. In this paper, we test how well the proposed numerical algorithm works on sparse matrices and large datasets by solving PageRank problems.
\begin{table}[H]
\small
  \centering
  \caption{Numerical experiments on the Pagerank problem for some data sets, tol=1e-8, Inner Iterations=50}%
  \label{b5.2}
    \scalebox{0.8}{\begin{tabular}{p{3.7cm}|p{4.7cm}|p{1.7cm}|p{1.7cm}}
     \hline
      Datasets   &Method&CPU(s)&Outer Iterations
    \\ \hline
    \multirow{8}{*}{\tt eu-2005}
    & GMRES-sh\newline FOM-sh
    & 90.25\newline 77.90
    & 11\newline 8
    \\ \cline{2-2}
    & \textbf{RandGMRES-sh}\newline \textbf{RandFOM-sh}
    & \textbf{51.14}\newline \textbf{43.90}
    & \textbf{11}\newline \textbf{9}
    \\ \cline{2-4}
    & GMRES-DR-sh\newline FOM-DR-sh
    & 61.12\newline 65.55
    & 8\newline 8
    \\ \cline{2-2}
    & \textbf{RandGMRES-DR-sh}\newline \textbf{RandFOM-DR-sh}
    & \textbf{32.02}\newline \textbf{31.01}
    & \textbf{8}\newline \textbf{8}
    \\ \hline

   \multirow{8}{*}{\tt Web-Google}
   & GMRES-sh\newline FOM-sh
   & 49.13\newline 58.74
   & 5\newline 5
   \\ \cline{2-2}
   & \textbf{RandGMRES-sh}\newline \textbf{RandFOM-sh}
   & \textbf{27.97}\newline \textbf{35.89}
   & \textbf{5}\newline \textbf{5}
   \\ \cline{2-4}
   & GMRES-DR-sh\newline FOM-DR-sh
   & 54.64\newline 53.84
   & 6\newline 6
   \\ \cline{2-2}
   & \textbf{RandGMRES-DR-sh}\newline \textbf{RandFOM-DR-sh}
   & \textbf{34.02}\newline \textbf{35.06}
   & \textbf{6}\newline \textbf{6}
    \\ \hline

   \multirow{8}{*}{\tt netherlands\_osm}
   & GMRES-sh\newline FOM-sh
   & 40.11\newline 45.27
   & 2\newline 2
   \\ \cline{2-2}
   & \textbf{RandGMRES-sh}\newline \textbf{RandFOM-sh}
   & \textbf{24.17}\newline \textbf{28.63}
   & \textbf{2}\newline \textbf{2}
   \\ \cline{2-4}
   & GMRES-DR-sh\newline FOM-DR-sh
   & 50.93\newline 51.63
   & 3\newline 3
   \\ \cline{2-2}
   & \textbf{RandGMRES-DR-sh}\newline \textbf{RandFOM-DR-sh}
   & \textbf{29.23}\newline \textbf{28.91}
   & \textbf{3}\newline \textbf{3}
    \\ \hline

   \multirow{8}{*}{\tt web-edu}
   & GMRES-sh\newline FOM-sh
   & 890.50\newline 710.87
   & 10\newline 8
   \\ \cline{2-2}
   & \textbf{RandGMRES-sh}\newline \textbf{RandFOM-sh}
   & \textbf{405.90}\newline \textbf{545.93}
   & \textbf{7}\newline \textbf{8}
   \\ \cline{2-4}
   & GMRES-DR-sh\newline FOM-DR-sh
   & 686.20\newline 740.74
   & 8\newline 8
   \\ \cline{2-2}
   & \textbf{RandGMRES-DR-sh}\newline \textbf{RandFOM-DR-sh}
   & \textbf{393.23}\newline \textbf{335.20}
   & \textbf{8}\newline \textbf{8}
    \\ \hline
    \end{tabular}}%
\end{table}%

We obtain the network link matrix $A$, then perform column normalization to derive matrix $Q$. Denoting the Google matrix as $G$, the modified PageRank formulation $G\mathbf{r}=\mathbf{r}$ \cite{25} is guaranteed to possess a unique positive solution, which is equivalent to solving the linear system denoted as problem $$(I-\alpha Q)\mathbf{r}=\frac{\mathbf{e}}{n}.$$

In many applications, we often encounter PageRank problems involving multiple damping factors. For instance, computing multiple PageRank vectors with different $\alpha$ values appears promising for designing anti-spam mechanisms. Furthermore, the Randomized alpha PageRank (RAPr) model requires solving a series of deterministic PageRank problems with numerous damping factors approaching 1 \cite{10}.

Table \ref{b5.2} compares the performance of eight algorithms for solving PageRank problems. The parameter $\alpha$ in the table represents different shift parameters, with the target accuracy set to $10^{-6}$ and the number of inner iterations fixed at 50. The decimal values listed under each shift parameter denote the actual accuracy achieved, all computed using the 1-norm of the true residual. The results clearly show that using the randomized algorithm reduces computational time while maintaining the same accuracy level. For easier observation, Figure \ref{t5.1} is presented in this work, using the "eu-2005" matrix with a damping factor of 0.99 as an example. The inner curves show how the residual norms change during the iteration processes of different algorithms when solving multiple shifted systems.

\begin{figure}[H]\caption{Convergence Behavior of Algorithms under Different $\alpha$ Values in eu-2005 }
\label{t5.1}
\centering
\begin{minipage}{0.49\linewidth}
\centering
  \includegraphics[width=1\textwidth]{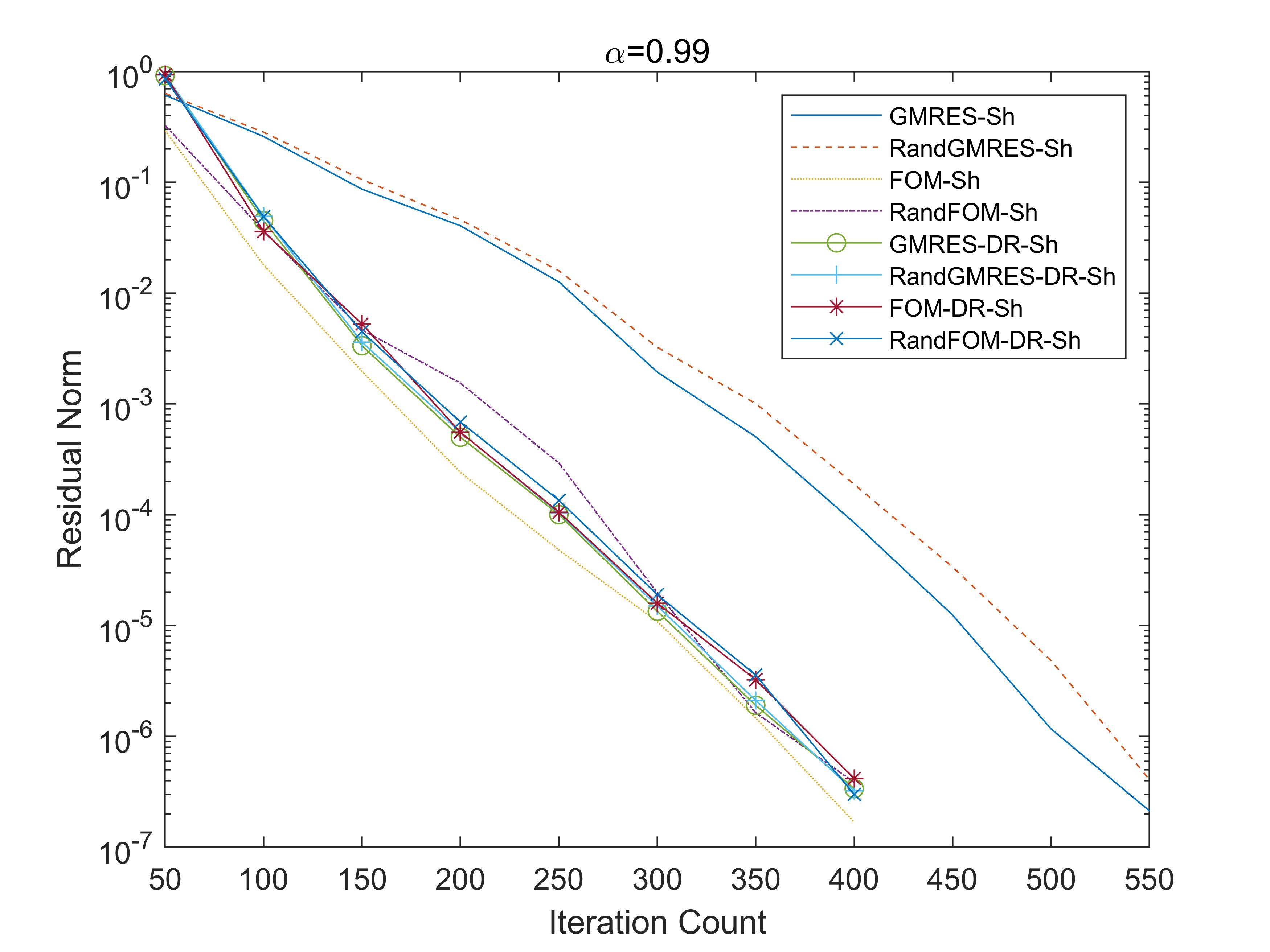}

\end{minipage}
\begin{minipage}{0.49\linewidth}
\centering
  \includegraphics[width=1\textwidth]{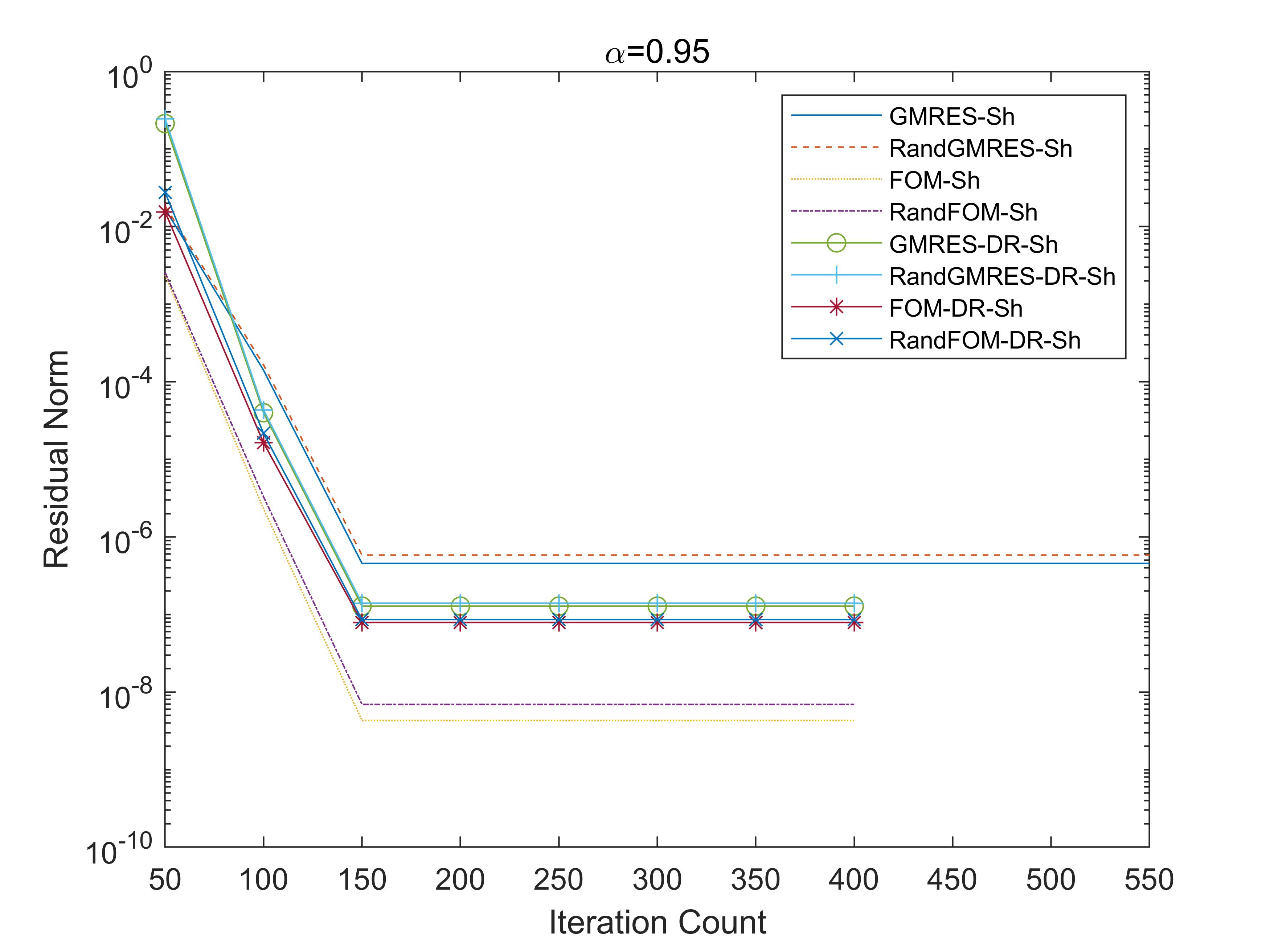}
 
\end{minipage}
\qquad

\begin{minipage}{0.49\linewidth}
\centering
  \includegraphics[width=1\textwidth]{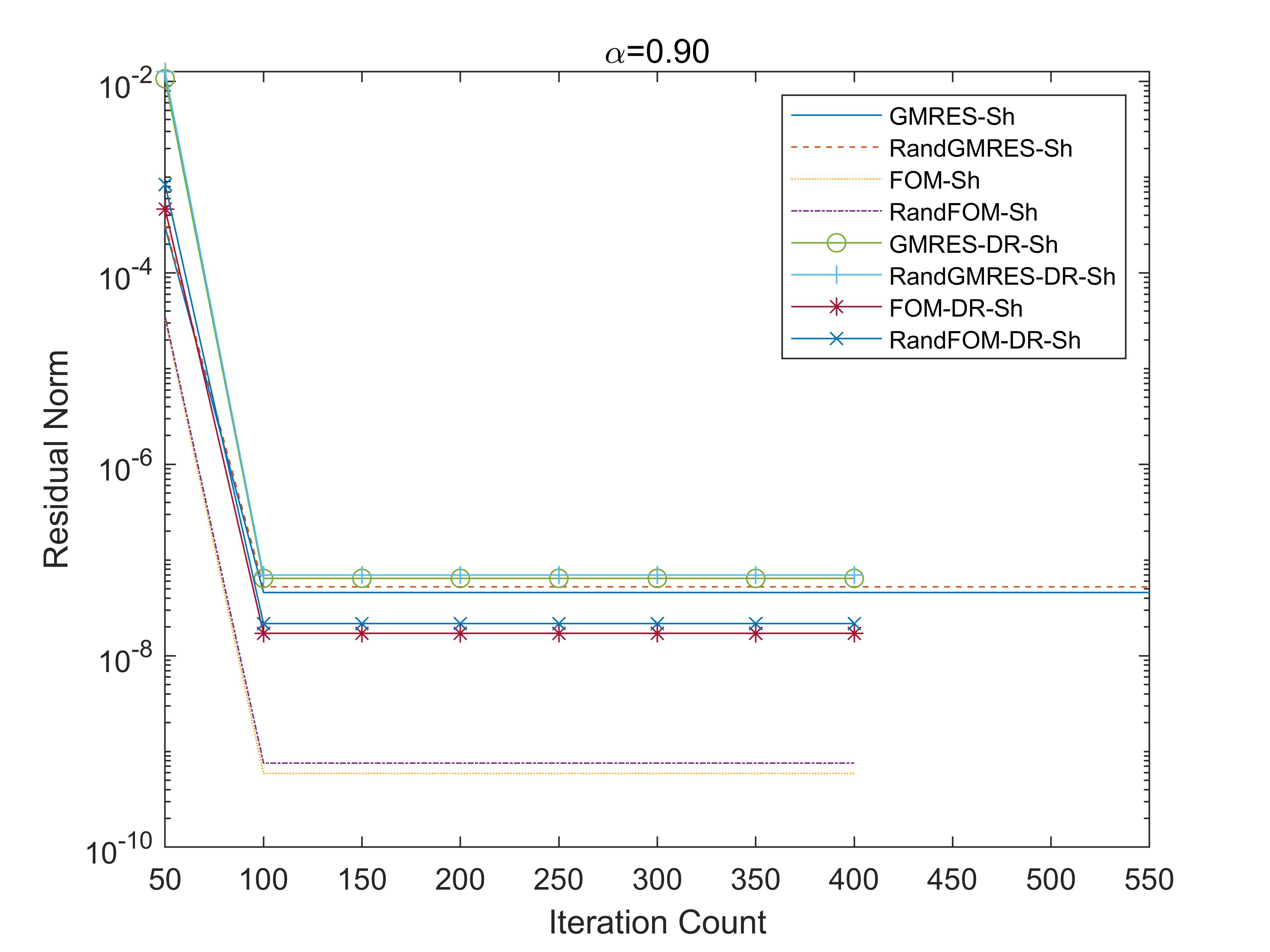}

\end{minipage}
\begin{minipage}{0.49\linewidth}
\centering
  \includegraphics[width=1\textwidth]{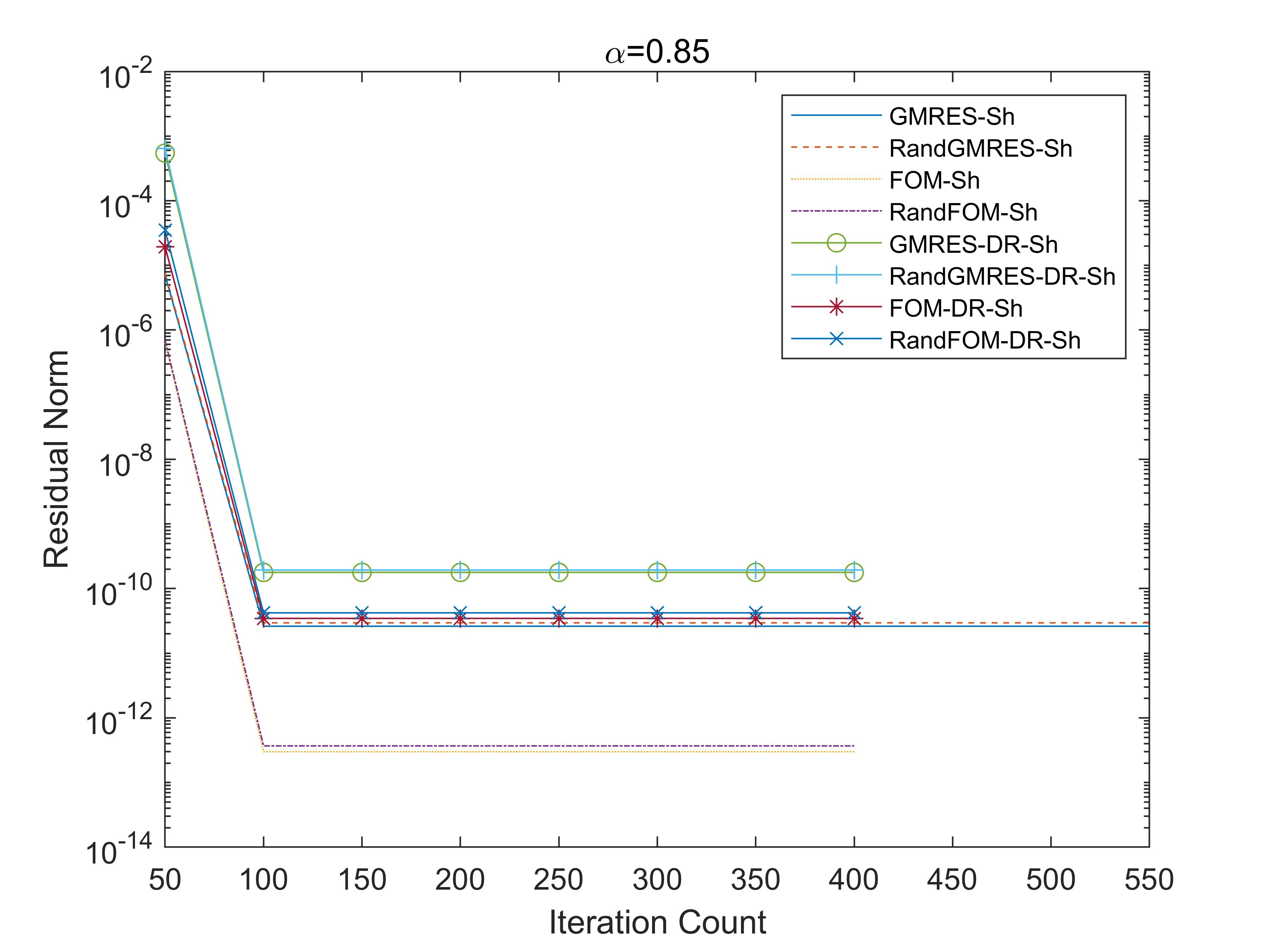}

\end{minipage}
\end{figure}

\subsection{Quantum Chromodynamics Problem}
Quantum Chromodynamics, \\the theory that describes the strong interaction between quarks and gluons, is a classic example in physics for solving "multi-shift equations." Here, "shift" can be understood as the dynamic variables of quantum fields in spacetime, and the "equations" refer to highly nonlinear equations of motion, such as the Yang-Mills equations. Because of phenomena like quark confinement and the strong interaction strength in the low-energy range, analytical solutions are almost impossible to get. This makes it an inherently infinite-dimensional, strongly coupled complex system. To solve this problem, physicists developed a numerical method called "lattice QCD."
The shifted problem can be stated as solving a set of shifted linear systems \cite{3,5,11} $$(D+\alpha_iI)\mathbf{x}_i=\mathbf{b},$$ where all systems use the same matrix $D$. In lattice QCD, the Dirac operator $D$ is usually large-scale, sparse, complex, and non-Hermitian. If we use iterative methods (such as CG or GMRES) separately for each shift parameter $\alpha_i$, the computational cost will be too high to afford. The shifted method solves this problem by making use of the fact that Krylov subspaces do not change under shift transformations. This method allows us to generate basis vectors that can be shared, so we can approximate solutions for all shifted systems at the same time in a single iterative process. This greatly reduces the computational burden.
Instead of giving the parameter $\sigma_i$ directly, it is of common use in QCD to write down its inverse $k_i=\frac{1}{\alpha_i}$.  In our experiments we use values
 of $m$ corresponding to the three values \cite{3}
 $$k_1=0.17,\ k_2=0.176,\ k_3=0.18.$$
 
 The matrices presented in the Table \ref{b5.3} are sourced from MatrixMarket, all being complex non-Hermitian matrices of order 49,152. The parameter $k$ in the table represents different shift parameters, with a target accuracy set to $10^{-8}$, and the number of inner iterations is set to 50. The decimal values listed below each shift parameter indicate the actual achieved accuracy, all computed using the 2-norm of the true residual.

\begin{table}[H]
\small
  \centering
  \caption{Numerical experiments on the QCD problem for some data sets, tol=1e-8, Inner Iterations=50}%
  \label{b5.3}
    \begin{tabular}{p{3.7cm}|p{4.0cm}|p{1.7cm}|p{1.7cm}}
     \hline
      Datasets   &Method&CPU(s)
      &Outer Iterations
    \\ \hline
    \multirow{8}{*}{\tt conf5\_4-8x8-15}
    & GMRES-sh\newline FOM-sh
    & 23.01\newline 40.05
    & 32\newline 54
    \\ \cline{2-2}
    & \textbf{RandGMRES-sh}\newline \textbf{RandFOM-sh}
    & \textbf{19.54}\newline \textbf{34.11}
    & \textbf{34}\newline \textbf{49}
    \\ \cline{2-4}
    & GMRES-DR-sh\newline FOM-DR-sh
    & 10.61\newline 12.23
    & 12\newline 13
    \\ \cline{2-2}
    & \textbf{RandGMRES-DR-sh}\newline \textbf{RandFOM-DR-sh}
    & \textbf{8.44}\newline \textbf{8.06}
    & \textbf{13}\newline \textbf{13}
    \\ \hline

   \multirow{8}{*}{\tt conf5\_4-8x8-05}
   & GMRES-sh\newline FOM-sh
   & 25.05\newline 34.02
   & 36\newline 45
   \\ \cline{2-2}
   & \textbf{RandGMRES-sh}\newline \textbf{RandFOM-sh}
   & \textbf{22.15}\newline \textbf{25.91}
   & \textbf{40}\newline \textbf{40}
   \\ \cline{2-4}
   & GMRES-DR-sh\newline FOM-DR-sh
   & 10.27\newline 11.73
   & 12\newline 13
   \\ \cline{2-2}
   & \textbf{RandGMRES-DR-sh}\newline \textbf{RandFOM-DR-sh}
   & \textbf{7.18}\newline \textbf{8.11}
   & \textbf{15}\newline \textbf{13}
    \\ \hline

   \multirow{8}{*}{\tt conf5\_4\_00l8x8\_1500} 
   & GMRES-sh\newline FOM-sh
   & 22.86\newline 41.79
   & 33\newline 53
   \\ \cline{2-2}
   & \textbf{RandGMRES-sh}\newline \textbf{RandFOM-sh}
   & \textbf{19.04}\newline \textbf{33.51}
   & \textbf{34}\newline \textbf{46}
   \\ \cline{2-4}
   & GMRES-DR-sh\newline FOM-DR-sh
   & 9.87\newline 12.52
   & 15\newline 13
   \\ \cline{2-2}
   & \textbf{RandGMRES-DR-sh}\newline \textbf{RandFOM-DR-sh}
   & \textbf{7.64}\newline \textbf{8.18}
   & \textbf{16}\newline \textbf{13}
    \\ \hline
    \end{tabular}%
\end{table}%
Our experimental results show limited performance improvement, which we attribute to the relatively small matrix dimension (below 50,000) used in the tests-likely constrained by the maximum available data scale for this category. Furthermore, the extent of improvement varies across matrices of different types and sizes. Figure \ref{p3} illustrates the convergence curves of several algorithms for matrix FEM\underline{~}3D\underline{~}thermal2 under a fixed shift.
\begin{figure}
  \centering
  \includegraphics[width=0.8\textwidth]{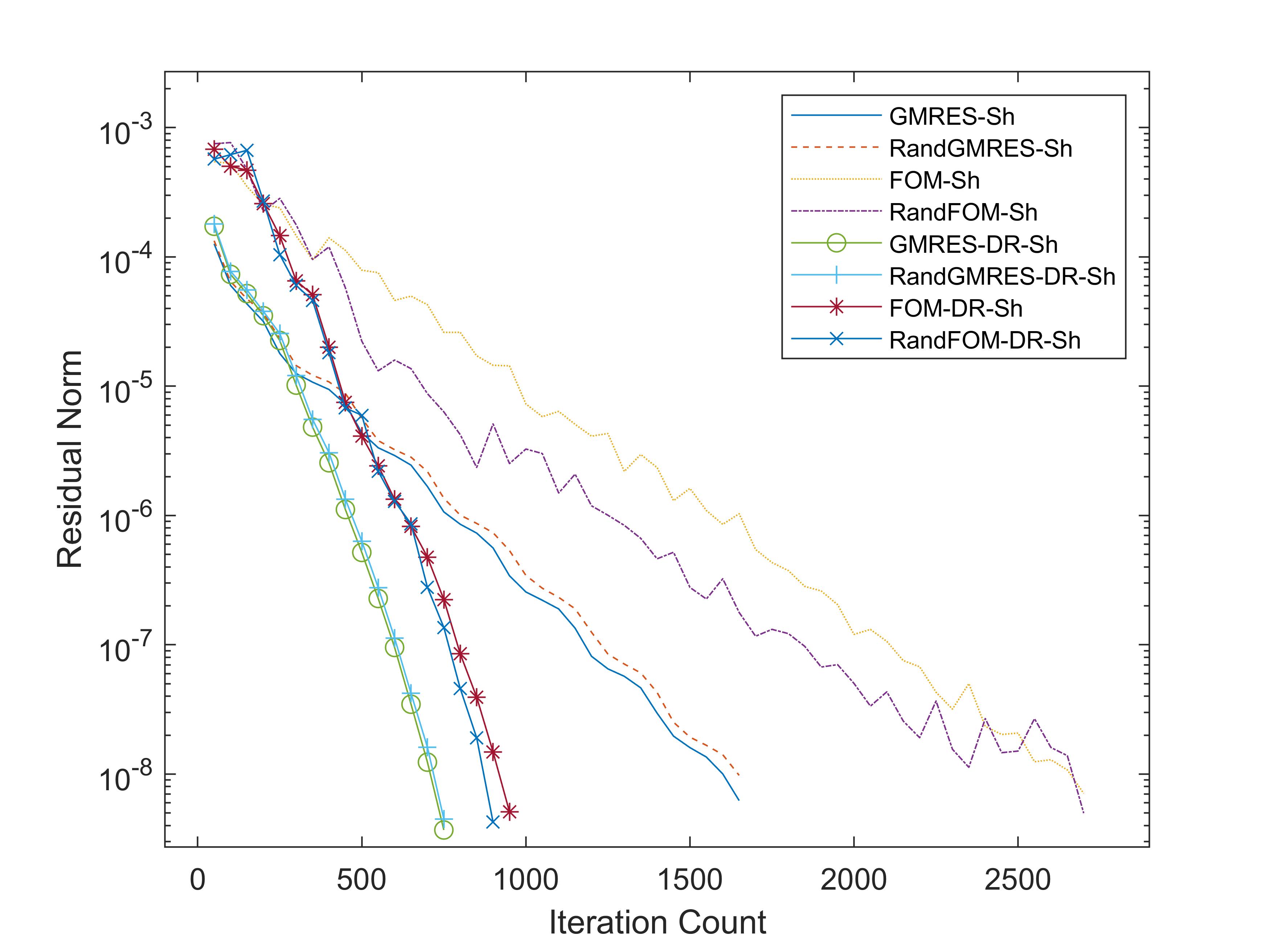}
  \caption{Convergence curve of dataset conf5\_4\_00l8x8\_1500}\label{p3}
\end{figure}

\subsection{Large-Scale Matrix Exponential Computation Problem}
In this section, we focus on the numerical approximation of the matrix exponential applied to a single vector 
\begin{equation}\label{E1}
  Z(t)=\exp(A)\mathbf{b},
\end{equation} where $t \in \mathbb{R}$ is a fixed constant, $tA \in \mathbb{R}^{n \times n}$ is an n-by-n large matrix with $\left \| tA \right \|  \gg 1$, and $\mathbf{b} \in \mathbb{R}^{n }$. Computing (\ref{E1}) is an important part of calculating exponential integrators, time-dependent partial differential equations \cite{34}, and approximating dynamical systems, among other things \cite{33}. We solve the problem of computing the matrix exponential acting on a single vector, $\exp(tA)\mathbf{b}$, using a rational approximation framework together with an efficient iterative solver. 
This method has three main parts: deriving a rational approximation, rewriting it into a numerical problem that is easier to solve, and developing a special iterative solver. Our approach is based on approximating the exponential function $e^x$ on the negative real axis using a rational function derived from the Carath{\'e}odory-Fej{\'e}r (CF) method \cite{35}. This technique gives a near-best rational approximation in the form $$e^x\approx \tilde{f} ^*(x)=w_0+\sum_{i=1}^{v} \frac{w_i}{x-\tau _i}, $$ where $\{\tau _i  \}^v_{i=1}\subset \mathbb{C}$ are the poles and $\{w_i \}^v_{i=1}\subset \mathbb{C}$ are the corresponding residues of the CF approximation of type(v,v). For a MATLAB function to compute the poles and residues of the CF method for the type(v,v) near-best approximation, we refer to \cite{35}. The order $v$ is typically chosen as 14. Consequently, the action of the matrix exponential is approximated by:
\begin{equation}\label{E2}
  exp(A)b\approx w_0\mathbf{b}+\sum_{i=1}^{v}  w_i(A-\tau _iI)^{-1}\mathbf{b}.
\end{equation}
We rewrite (\ref{E2}) as
\begin{equation}\label{E3}
  exp(A)b\approx w_0\mathbf{b}+A^{-1}\sum_{i=1}^{v}  w_i(I-\tau _iA^{-1})^{-1}\mathbf{b},
\end{equation}
 instead of solving (\ref{E2}), we propose to solve the following alternatively shifted linear systems with single right-hand sides $$(I-\tau _iA^{-1})\mathbf{x}_i=\mathbf{b},\ i=1,2,\dots,l,$$  and then collect the sum terms in (\ref{E3}) to get an approximation to $exp(A)B$.
\begin{table}[H]
\setlength{\tabcolsep}{30pt}
\centering
\caption{The first experiment, tol=1e-6, Inner Iterations=50}
\label{tb5.1111}
\scalebox{0.9}{\begin{tabular}{p{3.9cm}|p{1.7cm}|p{1.7cm}}
\hline  Method& CPU &Outer Iterations\\
\bottomrule[1pt] 
 GMRES-Sh & -  &-\\
 \textbf{RandGMRES-Sh} & \textbf{-} &\textbf{-}\\
 FOM-Sh & 	131.01s  &29 \\
 \textbf{RandFOM-Sh} & \textbf{44.91s}  &\textbf{19}\\
 GMRES-DR-Sh & -   &-\\
 \textbf{RandGMRES-Sh-DR} & \textbf{-}  &\textbf{-}\\
 FOM-DR-Sh & 113.08s  &26\\
 \textbf{RandFOM-DR-Sh} & \textbf{48.16s}  &\textbf{21}\\
\hline
\end{tabular}}
\end{table}
The first experiment is taken from \cite{31} and the matrix is provided by Al-Mohy and Higham \cite{32}; it is the block tridiagonal matrix $A = -2500 \times \beta \times \text{gallery}('poisson',396)$, where $A \in \mathbb{R}^{156816 \times 156816}$ is a multiple of the standard finite difference discretization of the 2D Laplacian with $\beta=0.02$. We aim to compute $e^{A}\mathbf{b}$ with the right-hand side $\mathbf{b} =\text{randn}(156816,1)$, but we focus only on the single right-hand side problem here. Table \ref{tb5.1111} presents the numerical results.

The second one is a variation of the example provided in \cite{40}
\\
\\
 D2=(N+1)/2*gallery(tridiag, N);\\
 I=speye(N);\\
 D2=kron(I,D2) +kron(D2, I);\\
 o=ones(N, 1);\\
 D1=(N+1)/2*spdiags([-o, 0*o,o],-1:1,N,N);\\
 D1 =kron(I, D1) +kron(D1, I);\\
 A =D2+v*D1;\\
 B=randn(N2, 1);\\
 
The matrices A correspond to the standard finite difference discretization of a 2D convection-diffusion equation on the unit square. This equation has a constant convection field and different convection parameters v. Here, we set $N=300$. Table \ref{tb5.11111} presents the numerical results.

The parameter $\tau$ represents different shift parameters, with the target accuracy set to $10^{-6}$ and the number of inner iterations fixed at 50. Note that the symbol '-' in the result tables means the method did not converge within 500 restarts. From the results in Tables \ref{tb5.1111} and \ref{tb5.11111}, we can see that GMRES-Sh and its variants failed to converge for both types of problems. This shows the inherent limitations of the GMRES-type framework when solving shifted linear systems. Figure \ref{tbep} illustrates the convergence curves of several algorithms for the second problem in large-scale matrix exponential computation under a fixed shift.
\begin{table}[H]
\setlength{\tabcolsep}{30pt}
\centering
\caption{The second experiment, tol=1e-6, Inner Iterations=50}
\label{tb5.11111}
\scalebox{0.9}{\begin{tabular}{p{3.9cm}|p{1.7cm}|p{1.7cm}}
\hline  Method & CPU&Outer Iterations\\
\bottomrule[1pt] 
 GMRES-Sh & -  &- \\
 \textbf{RandGMRES-Sh} & \textbf{-}  &\textbf{-} \\
 FOM-Sh & 	552.52s  & 189\\
 \textbf{RandFOM-Sh} & \textbf{105.57s}  &\textbf{63}\\
 GMRES-DR-Sh & -  &- \\
 \textbf{RandGMRES-Sh-DR} & \textbf{-}  &\textbf{-} \\
 FOM-DR-Sh & 85.67s  &29 \\
 \textbf{RandFOM-DR-Sh} & \textbf{40.48s}  &\textbf{26} \\
\hline
\end{tabular}}
\end{table}
\begin{figure}
  \centering
  \includegraphics[width=0.7\textwidth]{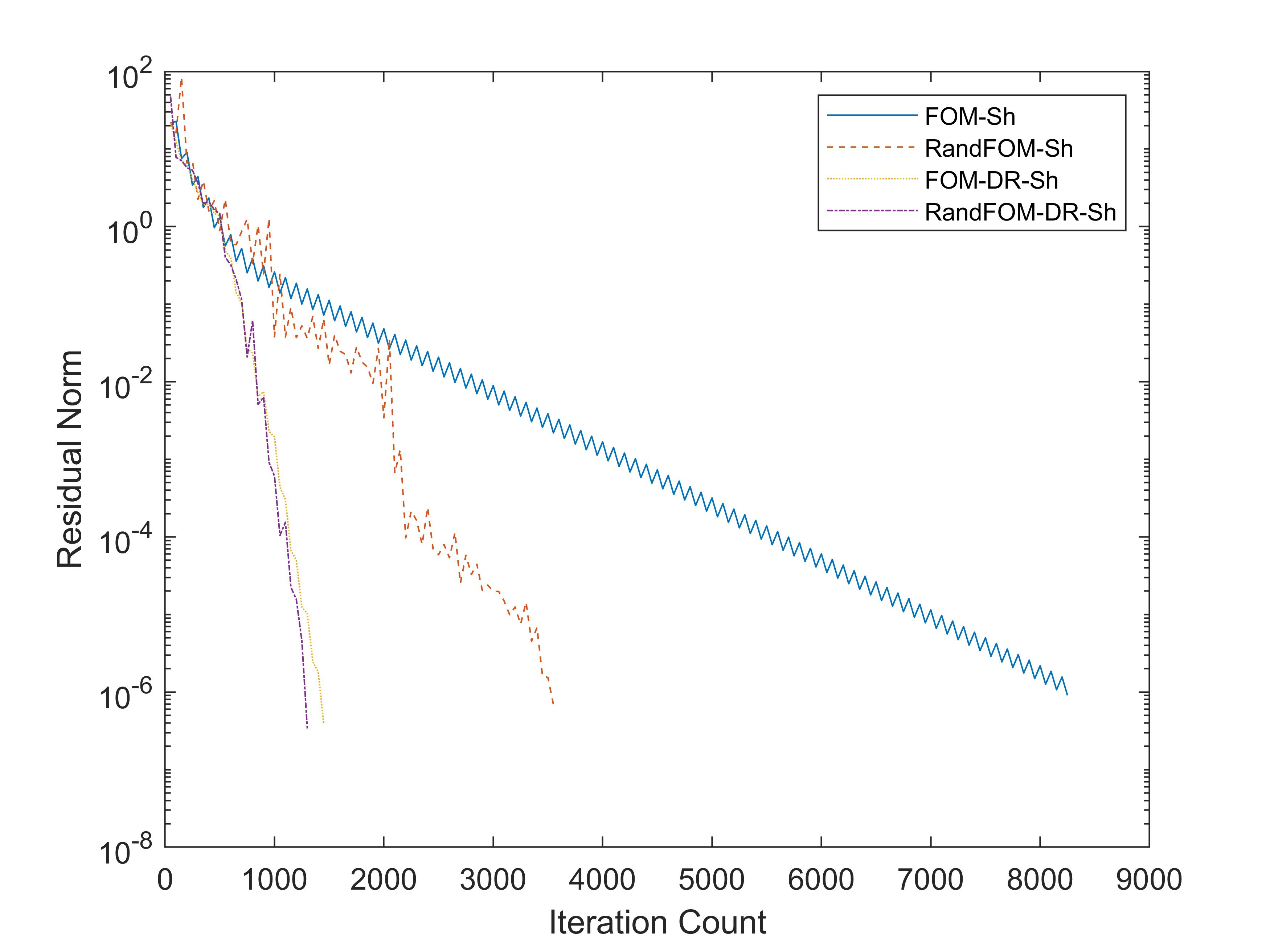}
  \caption{Convergence curves for large-scale matrix exponential computation}\label{tbep}
\end{figure}

In contrast, FOM-Sh and its variants show more consistent convergence behavior and more predictable performance. They still work well even when dealing with indefinite matrices or mixed positive and negative shifts, which proves they are much more robust. In addition, after adding sketching techniques, the computational efficiency is improved by at least two times. This shows that the proposed algorithms have a clear advantage over standard methods in solving large-scale shifted linear systems.
\section{Concluding remarks}
This work introduces four novel multi-shift algorithms incorporating randomized sketching techniques: RandGMRES-Sh, RandFOM-Sh, RandGMRES-DR-Sh, and RandFOM-DR-Sh. These are variants of GMRES-Sh, FOM-Sh, GMRES-DR-Sh, and FOM-DR-Sh, respectively, integrated with sketching. Moreover, the RandGMRES-DR and RandFOM-DR methods proposed in Section \ref{sec4} preserve the Arnoldi relation while maintaining computational costs comparable to currently popular deflated-restarting randomized algorithms. Our numerical experiments demonstrate that the proposed algorithms reduce the runtime by approximately half while maintaining a comparable number of iteration steps compared to the original methods (actual time savings vary depending on the problem being solved).

Furthermore, in terms of theoretical analysis, this paper proves that the two proposed methods for solving multi-shift systems of equations, namely RandGMRES-Sh and RandFOM-Sh, are also projection methods. We establish the relationship between the solutions and residuals of RandGMRES-Sh and RandFOM-Sh, along with the connection between RandGMRE S-Sh and GMRES-Sh, as well as between RandFOM-Sh and FOM-Sh. Finally, we prove the equivalence of the solution spaces of RandGMRES and sGMRES.
\section{Future Work and Outlook}
In addressing complex matrix problems, the performance of this method falls short of expectations. This is due to limitations in the available matrix dimensions on one hand, and the inherent properties of the random Sketching matrix on the other. However, in real matrix scenarios, the method performs relatively well. Additionally, we have observed that for problems related to dense matrices, the acceleration effect is often not significant enough. The main bottleneck lies in the computational cost of matrix-vector multiplication, rather than vector inner product. In the future, we aim to achieve further breakthroughs in these directions.

\end{document}